\documentclass[11pt]{amsart}
\usepackage[utf8]{inputenc}
\usepackage{amsmath}
\usepackage{amssymb}
\usepackage{amsthm}
\usepackage{tikz}
\usetikzlibrary{patterns,decorations.pathreplacing}
\usepackage{multirow}
\usepackage[all]{xy}
\usepackage[justification=centering]{caption}

\definecolor{vert}{RGB}{0,205,0}
\definecolor{aqua}{rgb}{0.0,1.0,1.0}

\newtheorem{theorem}{Theorem}[section]
\newtheorem{definitio}[theorem]{Definition}

\newtheorem{rem}[theorem]{Remark}
\newenvironment{remark}{\begin{rem} \rm }{\end{rem}}
\newtheorem{ex}[theorem]{Example}

\newtheorem{lemma}[theorem]{Lemma}
\newtheorem{proposition}[theorem]{Proposition}
\newtheorem{corollary}[theorem]{Corollary}

\newtheorem*{theo}{Theorem}

\newcommand{\Z}{\mathbb{Z}}

\newcommand{\R}{\mathbb{R}}

\newcommand{\CP}{\mathbb{CP}}
\newcommand{\RP}{\mathbb{RP}}
\newcommand{\HH}{\mathcal H}
\newcommand{\TT}{\mathcal T}

\newcommand{\fract}[2]{\hbox{\leavevmode\kern.1em \raise .25ex \hbox{\the\scriptfont0 $#1$}\kern-.1em }\big/{\hbox{\kern-.15em \lower .5ex \hbox{\the\scriptfont0 $#2$}}}}

\title{Multisections of surface bundles and bundles over $S^1$}
\author{Delphine Moussard}

\begin{document}

\begin{abstract}
 A multisection is a decomposition of a manifold into $1$--handlebodies, where each subcollection of the pieces intersects along a $1$--handlebody except the global intersection which is a closed surface. These generalizations of Heegaard splittings and Gay--Kirby trisections were introduced by Ben Aribi, Courte, Golla and the author, who proved in particular that any $5$--manifold admits such a multisection. In arbitrary dimension, we show that two classes of manifolds admit multisections: surface bundles and fiber bundles over the circle whose fiber itself is multisected. We provide explicit constructions, with examples.
 \medskip
 
 \noindent
 \textsc{2020 MSC Class:} 57M99, 57Q99, 57R99.
 \medskip
 
\end{abstract}

\maketitle

\tableofcontents

\section{Introduction}

Heegaard splittings are standard decompositions of closed orientable $3$--manifolds into two handlebodies. In \cite{GayKirby}, Gay and Kirby introduce analogous decompositions of $4$--manifolds, the so-called trisections; they proved that all closed orientable smooth $4$--manifolds admit such trisections. More recently, in \cite{BCGM}, Ben Aribi, Courte, Golla and the author studied a notion of multisection for closed orientable manifolds of any dimension, which encompasses Heegaard splittings and trisections: a multisection of an $(n+1)$--manifold is a decomposition into $n$ $1$--handlebodies such that each subcollection intersects along a $1$--handlebody, except the global intersection which is a closed surface. They proved in particular that any smooth $5$--manifold admits a multisection. Here, we study the existence of multisections in arbitrary dimension for two classes of manifolds: surface bundles and bundles over the circle.
This generalizes works on $4$--manifolds by Williams \cite{Williams} for surface bundles and by Koenig \cite{Koenig} for bundles over the circle.

\begin{theo}[Theorem~\ref{th:SurfaceBundle}]
 Any surface bundle admits a multisection.
\end{theo}

\begin{theo}[Theorem~\ref{th:bundleS1}]
 Any fiber bundle over $S^1$, whose fiber admits a multisection globally preserved by the monodromy ({\em ie} permuting the pieces of the multisection), admits a multisection.
\end{theo}

The proofs are constructive. We show how to get a multisection diagram of the bundle, either from a suitable decomposition of the base in the case of surface bundles, or from a multisection diagram of the fiber in the case of bundles over $S^1$. 

As noted by Koenig in \cite{Koenig}, for a $4$--dimensional fiber bundle over $S^1$, the fiber always admits a Heegaard splitting preserved by the monodromy: if $\HH$ is a given Heegaard splitting of the fiber and $\varphi$ is the monodromy, then, by uniqueness of Heegaard splittings, $\HH$ and $\varphi(\HH)$ have a common stabilization; this stabilization of $\HH$ is preserved by the monodromy. This argument applies to $5$--dimensional fiber bundles over $S^1$, because a given $4$--manifold (here, the fiber) always admits a trisection, which is unique up to stabilization \cite{GayKirby}. In higher dimensions however, the question of uniqueness up to stabilization of the multisections is open.

It may be worth mentionning that other notions of multisections do exist, that should not be confused with the ones we consider here. In \cite{RuTi}, Rubinstein and Tillmann define multisections of PL manifolds, which coincide with Heegaard splittings in dimension 3 and with trisections in dimension 4. In higher dimensions however, these multisections differ from ours; in particular they cannot be represented by diagrams. Also, in \cite{IN}, Islambouli and Naylor introduce multisections of $4$--dimensional manifolds, which generalize trisections to decompositions with an arbitrary number of pieces.

Throughout the article, all manifolds are orientable. We work in the PL category, but our constructions of multisections also work in the smooth category, {\em ie} with smooth manifolds and smooth multisections, see Remark~\ref{rem:smooth}. The only result which does not adapt to the smooth setting in all dimensions is Proposition~\ref{prop:destabilization} and this is precised at that point.


\section{Multisections and diagrams} 
\label{sec:defs}

For $n\geq3$, an {\em $n$--dimensional handlebody} is a manifold that admits a handle decomposition with one $0$--handle and some $1$--handles; the number of $1$--handles is the {\em genus} of the handlebody.

A {\em multisection}, or {\em $n$--section}, of a closed connected $(n+1)$--manifold $W$ is a decomposition $W=\cup_{i=1}^nW_i$ where:
\begin{itemize}
 \item for any non-empty proper subset $I$ of $\{1,\dots,n\}$, the intersection $W_I=\cap_{i\in I}W_i$ is a submanifold of $W$ which is PL--homeomorphic to an $(n-|I|+2)$--dimensional handlebody,
 \item $\Sigma=\cap_{i=1}^nW_i$ is a closed connected surface.
\end{itemize}
A multisection is a {\em Heegaard splitting} when $n=2$, a {\em trisection} when $n=3$, a {\em quadrisection} when $n=4$. 
The {\em genus} of the multisection is the genus of the {\em central surface} $\Sigma$.

\begin{remark} \label{rem:smooth}
 In the smooth category, the $W_I$ cannot all be smooth submanifolds of $W$: some corners necessarily appear. This forces us to give a more detailed definition, see~\cite{BCGM}.
\end{remark}

Given a $3$--dimensional handlebody $H$ with boundary $\Sigma$, a {\em cut-system} for $H$ is a collection of disjoint simple closed curves on $\Sigma$ which bound disjoint properly embedded disks in $H$ such that the result of cutting $H$ along these disks is a $3$--ball. Recall that a cut-system is well defined up to handleslides \cite[Corollary~1.6]{Johannson}.
A {\em multisection diagram} for a multisection as above is a tuple $\big(\Sigma;(c_i)_{1\leq i\leq n}\big)$, where $\Sigma$ is the central surface of the multisection and $c_i$ is a cut-system on $\Sigma$ for the $3$--dimensional handlebody $\cap_{j\neq i}W_j$. By \cite[Theorem 3.2]{BCGM}, a multisection diagram determines a unique PL--manifold (and a unique smooth manifold up to dimension $6$).

The connected sum of multisected manifolds can be performed around points of the central surfaces, so as to produce a multisection of the resulting manifold. This gives a simple way to define stabilizations of multisections: a {\em stabilization} of a multisection of an $(n+1)$--manifold $W$ is the connected sum of the multisected $W$ with the standard sphere $S^{n+1}$ equipped with a genus--$1$ multisection; the latter are given diagrammatically in Figure~\ref{fig:gen1spheres}. These stabilization moves can also be described as cut-and-paste operations on the multisection of $W$, see \cite{BCGM}.

\begin{figure}[htb] 
\begin{center}
\begin{tikzpicture} [scale=0.6]
\begin{scope} [scale=1.2]
 \draw (0,0) .. controls +(0,1) and +(-1,0) .. (3,1.5) .. controls +(1,0) and +(0,1) .. (6,0);
 \draw (0,0) .. controls +(0,-1) and +(-1,0) .. (3,-1.5) .. controls +(1,0) and +(0,-1) .. (6,0);
 \draw (2,0) ..controls +(0.5,-0.25) and +(-0.5,-0.25) .. (4,0);
 \draw (2.3,-0.1) ..controls +(0.6,0.2) and +(-0.6,0.2) .. (3.7,-0.1);
 \draw (1,-2) node {$S^{n+1}$};
 \foreach \x in {2.5,3.5} {
 \draw (\x,-0.15) ..controls +(0.2,-0.5) and +(0.2,0.5) .. (\x,-1.45);
 \draw[dashed] (\x,-0.15) ..controls +(-0.2,-0.5) and +(-0.2,0.5) .. (\x,-1.5);}
 \draw (3.05,-0.85) node {$\dots$};
 \draw (3,-1.6) node {$\underbrace{\phantom{aaaa}}$} (3,-2.2) node {$k$};
 \draw (3,0)ellipse(1.6 and 0.6);
 \draw (3,0)ellipse(2.6 and 1.2);
 \draw (5.1,0) node {$\dots$};
 \draw (5.1,0.1) node {$\overbrace{\phantom{aaaa}}$}; 
 \draw (5.1,0.4) -- (6,1) (6.3,1.3) node {$n-k$};
\end{scope}
\begin{scope} [xshift=10cm]
 \draw (0,0) .. controls +(0,1) and +(-1,0) .. (3,1.5) .. controls +(1,0) and +(0,1) .. (6,0);
 \draw (0,0) .. controls +(0,-1) and +(-1,0) .. (3,-1.5) .. controls +(1,0) and +(0,-1) .. (6,0);
 \draw (2,0) ..controls +(0.5,-0.25) and +(-0.5,-0.25) .. (4,0);
 \draw (2.3,-0.1) ..controls +(0.6,0.2) and +(-0.6,0.2) .. (3.7,-0.1);
 \draw (6.5,-2) node {$S^5$};
 \draw[green] (3.2,-0.2) ..controls +(0.2,-0.5) and +(0.2,0.5) .. (3.2,-1.5);
 \draw[dashed,green] (3.2,-0.2) ..controls +(-0.2,-0.5) and +(-0.2,0.5) .. (3.2,-1.5);
 \draw[blue] (3,-0.2) ..controls +(0.2,-0.5) and +(0.2,0.5) .. (3,-1.5);
 \draw[dashed,blue] (3,-0.2) ..controls +(-0.2,-0.5) and +(-0.2,0.5) .. (3,-1.5);
 \draw[red] (2.8,-0.2) ..controls +(0.2,-0.5) and +(0.2,0.5) .. (2.8,-1.5);
 \draw[dashed,red] (2.8,-0.2) ..controls +(-0.2,-0.5) and +(-0.2,0.5) .. (2.8,-1.5);
 \draw[orange] (3,0)ellipse(1.6 and 0.8);
\end{scope}
\begin{scope} [xshift=17cm]
 \draw (0,0) .. controls +(0,1) and +(-1,0) .. (3,1.5) .. controls +(1,0) and +(0,1) .. (6,0);
 \draw (0,0) .. controls +(0,-1) and +(-1,0) .. (3,-1.5) .. controls +(1,0) and +(0,-1) .. (6,0);
 \draw (2,0) ..controls +(0.5,-0.25) and +(-0.5,-0.25) .. (4,0);
 \draw (2.3,-0.1) ..controls +(0.6,0.2) and +(-0.6,0.2) .. (3.7,-0.1);
 \draw[blue] (3.2,-0.2) ..controls +(0.2,-0.5) and +(0.2,0.5) .. (3.2,-1.5);
 \draw[dashed,blue] (3.2,-0.2) ..controls +(-0.2,-0.5) and +(-0.2,0.5) .. (3.2,-1.5);
 \draw[red] (3,-0.2) ..controls +(0.2,-0.5) and +(0.2,0.5) .. (3,-1.5);
 \draw[dashed,red] (3,-0.2) ..controls +(-0.2,-0.5) and +(-0.2,0.5) .. (3,-1.5);
 \draw[green] (3,0)ellipse(1.6 and 0.8);
 \draw[orange] (3,0)ellipse(1.8 and 1);
\end{scope}
\end{tikzpicture}
\caption{Genus--$1$ multisection diagrams of the spheres, where $0<k<n$} \label{fig:gen1spheres}
\end{center}
\end{figure}

\begin{proposition} \label{prop:destabilization}
 Let $W$ be an $(n+1)$--manifold with a multisection diagram $\big(\Sigma;(c_i)_{1\leq i\leq n}\big)$. Assume the diagram contains two groups of $k$ and $n-k$ parallel curves, with $0<k<n$, such that a curve of one group meets a curve of the other group in exactly one point. Then the associated multisection is the result of a stabilization.
\end{proposition}
\begin{proof}
 Denote by $\alpha_1,\dots,\alpha_k$ and $\beta_{k+1},\dots,\beta_n$ the two groups of curves. Since two curves in the same collection $c_i$ have to be disjoint but not parallel, there is one of these curves in each $c_i$. If a curve of the diagram, other than the $\beta$--curves, meets the $\alpha$--curves, then it is in the same collection $c_i$ as one of the $\beta$--curves, so that it can be slid along this $\beta$--curve until it gets disjoint from the $\alpha$--curves. The same is true for a curve, other than the $\alpha$--curves, that would meet the $\beta$--curves. Hence, after handleslides, we can assume that the $\alpha$--curves and $\beta$--curves are disjoint from any other curve of the diagram. Now a regular neighborhood of the $\alpha$--curves and the $\beta$--curves in $\Sigma$ is a punctured torus. It implies that the diagram is a connected sum of a lower-genus diagram with a diagram of Figure~\ref{fig:gen1spheres}. Since the multisected manifold is determined by the diagram, it is the result of a stabilization.
\end{proof}

We will use this proposition in Section~\ref{sec:bundlesS1} to simplify some multisection diagrams. We may note that, for a smooth manifold, this proof works up to dimension $6$ only: the point is that a multisection diagram determines a unique smooth manifold up to dimension $6$ only, so that the last sentence of the proof does not apply in higher dimensions.

\section{Multisecting surface bundles}
\label{sec:surfacebundles}

Surface bundles turn out to admit multisections in any dimension. A multisection of a surface bundle can be obtained from a suitable decomposition of the base into balls. This generalizes the work of Gay--Kirby \cite{GayKirby} and Williams \cite{Williams} on trisections of $4$--dimensional surface bundles.

Let $M$ be a closed $d$--manifold. We call {\em good ball decomposition} of the manifold $M$ a decomposition $M=\cup_{i=0}^dM_i$ where, for all non-empty $I\subset\{0,\dots,d\}$, $\cap_{i\in I}M_i$ is a disjoint union of embedded $(d-|I|+1)$--balls. Such decompositions do exist. 

\begin{lemma}
 Any closed manifold admits a good ball decomposition.
\end{lemma}
\begin{proof}
 Let $M$ be a closed $d$--manifold and $\TT$ a triangulation of $M$. We consider the first and second barycentric subdivisions $\TT^1$ and $\TT^2$ of $\TT$. For $i\in\{0,\dots,d\}$, let $V_i$ be the set of barycenters of $i$--faces of $\TT$; note that the union of the $V_i$ is the set of vertices of $\TT^1$. We define $M_i$ as the union of the stars of the vertices of $V_i$ in $\TT^2$, see Figure~\ref{figGBD}. 
 We may note that $M_i$ is the union of $i$--handles in the handle decomposition of $M$ associated with the triangulation $\TT$ (see for instance \cite[Proposition~6.9]{RS}). The condition on the intersections is direct. 
\end{proof}

\begin{figure}[htb]
\begin{center}
\begin{tikzpicture} [yscale=0.87,scale=0.5]
\foreach \s in {-1,1} {
\begin{scope} [xscale=\s]
 \draw (0,0) -- (6,0) -- (0,12) -- (0,0);
 \draw (6,0) -- (0,4) -- (3,6);
 \draw (0,0) -- (3,2) -- (3,6) -- (0,8) (0,2) -- (6,0) -- (1.5,5) -- (0,12) (3,0) -- (0,4) -- (4.5,3) (0,4) -- (1.5,9);
 \draw[red!60,fill=red!60,opacity=0.5] (3,0) -- (6,0) -- (4.5,3) -- (3,3.33) -- (3,2) -- (2,1.33) -- (3,0) (0,12) -- (0,8) -- (1,7.33) -- (1.5,9) -- (0,12);
 \draw[vert!60,fill=vert!60,opacity=0.5] (0,0) -- (3,0) -- (2,1.33) -- (0,2) -- (0,0) (1.5,9) -- (1,7.33) -- (1.5,5) -- (3,3.33) -- (4.5,3) -- (1.5,9);
 \draw[blue!60,fill=blue!60,opacity=0.5] (0,2) -- (2,1.33) -- (3,2) -- (3,3.33) -- (1.5,5) -- (1,7.33) -- (0,8) -- (0,2);
\end{scope}}
 \foreach \x/\y in {-6/0,6/0,0/12}
 \draw[red] (\x,\y) node {$\bullet$};
 \draw[red] (6,0) node[right] {$V_0$};
 \draw[red!60] (6,2) node {$M_0$};
 \foreach \x/\y in {0/0,3/6,-3/6}
 \draw[vert] (\x,\y) node {$\bullet$};
 \draw[vert] (3,6) node[right] {$V_1$};
 \draw[vert!60] (3,8) node {$M_1$};
 \draw[blue] (0,4) node {$\bullet$} node[right] {$V_2$};
 \draw[blue!60] (0,6.5) node {$M_2$};
\end{tikzpicture} \caption{For $d=2$, decomposition of a $2$--simplex into the $M_i$} \label{figGBD}
\end{center}
\end{figure}

Let $p:W\to M$ be a surface bundle. Taking the preimage of a good ball decomposition of~$M$, we get a decomposition of $W$ into pieces that are products of a surface, the fibre, with balls. To obtain a multisection from this, we dig some ``disk tunnels'' into the different pieces, which we add to other pieces.
If $N\subset M$ is a disjoint union of contractible subspaces of~$M$, a {\em disk section of $p$ over $N$} is a submanifold $Z$ of $W$ such that $p^{-1}(x)\cap Z$ is a $2$--disk for all $x\in N$, and $p(Z)=N$.

\begin{theorem} \label{th:SurfaceBundle}
 Fix $n>1$. Let $p:W\to M$ be a surface bundle, where $W$ is a closed $(n+1)$--manifold. Assume we are given a good ball decomposition $M=\cup_{i=1}^nM_i$. Fix pairwise disjoint disk sections $Z_i$ of $p$ over $M_i$, for $1\leq i\leq n$. Set $W_i=\left(\overline{p^{-1}(M_i)\setminus Z_i}\right)\cup Z_{i+1}$, where the indices are considered modulo $n$. Then $W=\cup_{i=1}^nW_i$ is a multisection.
\end{theorem}

\begin{proof}
 Set $M_I=\cap_{i\in I}M_i$. 
 We slightly abuse notation by denoting $Z_i=M_i\times D_i$, where $D_i$ is a disk in the fiber $S$ above each point of $M_i$, indeed depending on this point. In this way, $W_i$ can be written as 
 $$W_i=\left(M_i\times\overline{S\setminus D_i}\right)\cup\big( M_{i+1}\times D_{i+1}\big).$$
 More generally, for $I\subset\{1,\dots,n\}$, one can check by induction on $|I|$ that:
 \begin{align*}
  W_I=&\ 
  \left(M_I\times\overline{S\setminus\cup_{i\in I}D_i}\right) 
  \cup\bigcup_{\substack{i\in I\\i+1\notin I}}\left(M_{I\cup\{i+1\}\setminus\{i\}}\times D_{i+1}\right)
  \cup\bigcup_{\substack{i\in I\\i+1\in I}}\left(M_{I\setminus\{i\}}\times\partial D_{i+1}\right). 
 \end{align*}
 We fix a non-empty proper subset $I$ of $\{1,\dots,n\}$ and check that $W_I$ is a handlebody. 
 In the first term, each connected component is a thickening of a surface with non-empty boundary, thus a handlebody. The second term adds $1$--handles to these handlebodies; indeed, $M_I$ and $M_{I\cup\{i+1\}\setminus\{i\}}$ are made of $(n-|I|)$--balls that intersect along $(n-|I|-1)$--balls. For $i\in I$ such that $i+1\in I$, $M_{I\setminus\{i\}}$ is made of $(n-|I|+1)$--balls and contains in its boundary the $(n-|I|)$--balls composing~$M_I$. Hence each component of $M_{I\setminus\{i\}}\times\partial D_{i+1}$ is a $D^{n-|I|+1}\times S^1$ glued along some $D^{n-|I|}\times S^1$, with $D^{n-|I|}$ living in the boundary of $D^{n-|I|+1}$, to some handlebodies in $M_I\times\overline{S\setminus\cup_{i\in I}D_i}$. This has the effect to glue together some of the latter handlebodies along a $1$--handle corresponding to the boundary component $\partial D_{i+1}$ of $\overline{S\setminus\cup_{i\in I}D_i}$. Finally, $W_I$ is made of handlebodies and it remains to check that it is connected. 
 
 Thanks to the above formula for $W_I$, it is enough to check that the different connected components of $M_I$ are connected by some number of paths, each contained in $M_{I\cup\{i+1\}\setminus\{i\}}$ (which is contained in $M_{I\setminus\{i\}}$) for some $i\in I$. Note that the connectedness of $M$ implies that $\cup_{|I|=n-1}M_I$ is connected (the good ball decomposition provides a CW--complex structure where the $k$--skeleton is $\cup_{|I|=n-k}M_I$). Hence two components of $M_I$ are always connected by a path in $\cup_{|I|=n-1}M_I$, and each interval of this path which is not in $M_I$ is a component of $M_{\{1,\dots,n\}\setminus\{i\}}$ for some $i\in I$, which is contained in $M_{I\cup\{i+1\}\setminus\{i\}}$.
 
 Finally, the central piece $\Sigma=W_{\{1,\dots,n\}}$ is given by
 $$\Sigma=\left(M_{\{1,\dots,n\}}\times\overline{S\setminus\cup_{i=1}^nD_i}\right)\cup\bigcup_{i=1}^n\left(M_{\{1,\dots,n\}\setminus\{i\}}\times\partial D_{i+1}\right).$$
 The first term is made of copies of $S$ with the open $D_i$ removed, while the second term is made of tubes joining the boundary components of these copies. Hence $\Sigma$ is a closed surface.
\end{proof}

\begin{remark}
 The multisection obtained in the theorem has genus $vg+e-v+1$, where $v$ is the number of points in $M_{\{1,\dots,n\}}$, $e$ is the number of intervals in $\cup_{|I|=n-1}M_I$, and $g$ is the genus of the fiber. 
\end{remark}

We now present some examples. When the bundle is simply a product, we define the disk sections as $Z_i=M_i\times D_i$, where the $D_i$ are disjoint disks on the fiber.

We start with bundles over spheres. The standard sphere admits the good ball decomposition given by the following lemma, see Figure~\ref{figdecspheres}. 

\begin{lemma} \label{lemmaDecomposeSphere}
 The $(n-1)$--sphere admits a good ball decomposition $S^{n-1}=\cup_{i=1}^nB_i$ where $\cap_{i\in I}B_i$ is a single $(n-|I|)$--ball for $I\subsetneq\{1,\dots,n\}$ and $\cap_{i=1}^nB_i$ is made of two points.
\end{lemma}
\begin{proof}
 Consider the map $\varphi:S^{n-1}\subset\R^{n}\to\R^{n-1}$ sending $(x_1,\dots,x_n)$ to $(x_1,\dots,x_{n-1})$. Its image $B^{n-1}$ can be viewed as an $(n-1)$--simplex and cut into the cones with vertex its center and bases its faces. The pull-back of this decomposition provides the required decomposition of~$S^{n-1}$.
\end{proof}

\begin{figure}[htb] 
\begin{center}
\begin{tikzpicture}
\begin{scope} [scale=0.4]
 \draw (0,0) circle (2);
 \foreach \s in {-1,1} {\draw (0,2*\s) node {$\bullet$};}
 \draw (0,-4.5) node {$S^1$};
\end{scope}
\begin{scope}[xshift=4cm,scale=0.6]
 \draw (0,0) circle (2);
 \foreach \s/\d in {-1/,1/dashed} {
 \draw (0,2*\s) node {$\bullet$};
 \draw[\d] (-2,0) .. controls +(1,0.8*\s) and +(-1,0.8*\s) .. (2,0);
 \draw[very thick] (0,-2) .. controls +(1.8*\s,1) and +(1.8*\s,-1) .. (0,2);}
 \draw[very thick,dashed] (0,-2) .. controls +(-0.5,1) and +(-0.5,-1) .. (0,2);
 \draw (0,-3) node {$S^2$};
\end{scope}
\begin{scope}[xshift=9cm,scale=0.6]
 \draw[gray!60,fill=gray!60] (-4,-1) -- (2.5,-1) -- (4,1) -- (-2.5,1) -- (-4,-1);
 \draw[fill=gray!20,opacity=0.8] (0,-2) arc (-90:90:2) .. controls +(-0.8,-1) and +(-0.8,1) .. (0,-2);
 \draw[fill=gray!45,opacity=0.8] (0,2) arc (90:270:2) .. controls +(0.8,1) and +(0.8,-1) .. (0,2);
 \draw[dashed] (-2,0) .. controls +(1,0.8) and +(-1,0.8) .. (2,0);
 \draw[dashed,fill=gray!80,opacity=0.8] (0,-2) .. controls +(0.8,1) and +(0.8,-1) .. (0,2) .. controls +(-0.8,-1) and +(-0.8,1) .. (0,-2);
 \draw (0,-2) .. controls +(0.8,1) and +(0.8,-1) .. (0,2);
 \draw[gray!60,fill=gray!60,opacity=0.8] (-2,0) .. controls +(1,-0.8) and +(-1,-0.8) .. (2,0) -- (2,-1) -- (-2,-1) -- (-2,0);
 \draw (-2,0) .. controls +(1,-0.8) and +(-1,-0.8) .. (2,0);
 \foreach \s in {-1,1} {
 \draw[scale=\s] (0.55,-0.55) node {$\bullet$};}
 \draw (0,-3) node {$S^3$};
\end{scope}
\end{tikzpicture}
\caption{Good ball decomposition of $S^k$ for small $k$} \label{figdecspheres}
\end{center}
\end{figure}

\begin{corollary}
 A surface bundle over $S^{n-1}$ with fiber a closed surface of genus $g$ admits a multisection of genus $2g+n-1$. 
\end{corollary}

If $W$ is a surface bundle over $S^{n-1}$, the associated multisection has a central surface given by two copies of the fiber (the preimages of $\cap_{i=1}^nB_i$) joined by a tube above each $\cap_{j\neq i}B_j$. 
Examples of the associated diagram are given in Figures~\ref{figS2bundles} and~\ref{figS1S1S3}. The cut-system associated to $W_{\{1,\dots,n\}\setminus\{i\}}$ is obtained as follows. A first curve is given by a meridian curve around the tube above $\cap_{j\neq i-1}B_j$. Then take a family of properly embedded arcs on $S\setminus\cup_{j\neq i} D_j$ that cut it into a disk, and build simple closed curves on $\Sigma$ by joining the two copies of these arcs on the two copies of the punctured $S$ by parallel arcs on the tubes.

\begin{figure}[htb] 
\begin{center}
\begin{tikzpicture} [scale=0.45]
\begin{scope}
 \draw[rounded corners=50pt] (0.5,0) -- (0.5,-5) -- (15.5,-5) -- (15.5,5) -- (0.5,5) -- (0.5,0);
 \foreach \x in {4,8,12} {
 \draw (\x+0.5,2) .. controls +(-0.5,-0.6) and +(0,1) .. (\x-0.5,0) .. controls +(0,-1) and +(-0.5,0.6) .. (\x+0.5,-2);
 \draw (\x,1.4) .. controls +(0.5,-0.6) and +(0,0.3) .. (\x+0.5,0) .. controls +(0,-0.3) and +(0.5,0.6) .. (\x,-1.4);}
 \foreach \x/\c in {0/vert,4/purple,8/blue,12/orange} {
 \draw[color=\c] (\x+0.5,0) .. controls +(1,-0.5) and +(-1,-0.5) .. (\x+3.5,0);
 \draw[color=\c,dashed] (\x+0.5,0) .. controls +(1,0.5) and +(-1,0.5) .. (\x+3.5,0);}
 \foreach \x/\c in {4/orange,8/vert,12/purple} {
 \draw[color=\c] (\x,0) ellipse (1.5 and 3);}
 \foreach \x/\c in {4/blue,8/orange,12/vert} {
 \draw[color=\c] (\x,0) ellipse (1.8 and 3.4);}
 \foreach \x/\c in {0/purple,4/blue} {
 \draw[rounded corners=40pt,color=\c] (\x+2,0) -- (\x+2,-4) -- (\x+10,-4) -- (\x+10,4) -- (\x+2,4) -- (\x+2,0);}
\end{scope}
\begin{scope} [xshift=15.75cm]
 \draw[rounded corners=50pt] (0.5,0) -- (0.5,-5) -- (19.5,-5) -- (19.5,5) -- (0.5,5) -- (0.5,0);
 \foreach \x in {4,8,12,16} {
 \draw (\x+0.5,2) .. controls +(-0.5,-0.6) and +(0,1) .. (\x-0.5,0) .. controls +(0,-1) and +(-0.5,0.6) .. (\x+0.5,-2);
 \draw (\x,1.4) .. controls +(0.5,-0.6) and +(0,0.3) .. (\x+0.5,0) .. controls +(0,-0.3) and +(0.5,0.6) .. (\x,-1.4);}
 \foreach \x/\c in {0/vert,4/aqua,8/blue,12/red,16/orange} {
 \draw[color=\c] (\x+0.5,0) .. controls +(1,-0.5) and +(-1,-0.5) .. (\x+3.5,0);
 \draw[color=\c,dashed] (\x+0.5,0) .. controls +(1,0.5) and +(-1,0.5) .. (\x+3.5,0);}
 \foreach \x/\c in {4/orange,8/vert,12/aqua,16/blue} {
 \draw[color=\c] (\x,0) ellipse (1.5 and 3.1);}
 \foreach \x/\c in {4/blue,8/red,12/orange,16/vert} {
 \draw[color=\c] (\x,0) ellipse (1.7 and 3.4);}
 \foreach \x/\c in {4/red,8/orange,12/vert,16/aqua} {
 \draw[color=\c] (\x,0) ellipse (1.3 and 2.7);}
 \foreach \x/\c in {0/aqua,4/blue,8/red} {
 \draw[rounded corners=40pt,color=\c] (\x+2.1,0) -- (\x+2.1,-4) -- (\x+9.9,-4) -- (\x+9.9,4) -- (\x+2.1,4) -- (\x+2.1,0);}
\end{scope}
\end{tikzpicture}
\caption{Multisection diagrams for $S^2\times S^3$ and $S^2\times S^4$} \label{figS2bundles}
\end{center}
\end{figure}

\begin{figure}[htb] 
\begin{center}
\begin{tikzpicture} [xscale=0.8,yscale=0.6]
 \draw (8,0) ellipse (7.5 and 6.5);
 \foreach \x in {4,8,12} {
 \draw (\x+0.5,2) .. controls +(-0.5,-0.6) and +(0,1) .. (\x-0.5,0) .. controls +(0,-1) and +(-0.5,0.6) .. (\x+0.5,-2);
 \draw (\x,1.4) .. controls +(0.5,-0.6) and +(0,0.3) .. (\x+0.5,0) .. controls +(0,-0.3) and +(0.5,0.6) .. (\x,-1.4);}
 \foreach \y in {-4.5,4.5} {
 \draw (6,\y+0.5) .. controls +(0.6,-0.5) and +(-1,0) .. (8,\y-0.5) .. controls +(1,0) and +(-0.6,-0.5) .. (10,\y+0.5);
 \draw (6.6,\y) .. controls +(0.6,0.5) and +(-0.3,0) .. (8,\y+0.5) .. controls +(0.3,0) and +(-0.6,0.5) .. (9.4,\y);}
 \foreach \x/\c in {0/vert,4/purple,8/blue,12/orange} {
 \draw[color=\c] (\x+0.5,0) .. controls +(1,-0.5) and +(-1,-0.5) .. (\x+3.5,0);
 \draw[color=\c,dashed] (\x+0.5,0) .. controls +(1,0.5) and +(-1,0.5) .. (\x+3.5,0);}
 \foreach \x/\c in {4/orange,8/vert,12/purple} {
 \draw[color=\c] (\x,0) ellipse (1.2 and 2.7);}
 \foreach \x/\c in {4/blue,8/orange,12/vert} {
 \draw[color=\c] (\x,0) ellipse (1.4 and 3);}
 \foreach \x/\c in {0/purple,4/blue} {
 \draw[rounded corners=40pt,color=\c] (\x+2.4,0) -- (\x+2.4,-3.3) -- (\x+9.6,-3.3) -- (\x+9.6,3.3) -- (\x+2.4,3.3) -- (\x+2.4,0);
 \draw[color=\c] (8,0) ellipse (7.1+\x/20 and 6.1+\x/20);}
 \foreach \x/\c in {0/orange,0.1/vert} {
 \draw[rounded corners=30pt,color=\c] (8,5.6+\x) -- (10.4+\x,5.6+\x) -- (10.4+\x,-5.6-\x) -- (5.6-\x,-5.6-\x) -- (5.6-\x,5.6+\x) -- (8,5.6+\x); }
 \foreach \s/\c in {1/purple,-1/blue} {
 \draw[color=\c] (8+1.2*\s,4.35) .. controls +(\s,-1) and +(0,1) .. (8+2.2*\s,0) .. controls +(0,-1) and +(\s,1) .. (8+1.2*\s,-4.35);
 \draw[color=\c,dashed] (8+1.2*\s,4.35) .. controls +(0.6*\s,-1) and +(0,1) .. (8+2.1*\s,0) .. controls +(0,-1) and +(0.6*\s,1) .. (8+1.2*\s,-4.35);}
 \foreach \s/\c in {-1/orange,1/vert} {
 \draw[color=\c] (8+\s,4.22) .. controls +(0.8*\s,-1) and +(0,1) .. (8+1.9*\s,0) .. controls +(0,-1) and +(0.8*\s,1) .. (8+\s,-4.22);
 \draw[color=\c,dashed] (8+\s,4.22) .. controls +(0.5*\s,-1) and +(0,1) .. (8+1.8*\s,0) .. controls +(0,-1) and +(0.5*\s,1) .. (8+\s,-4.22);}
\end{tikzpicture}
\caption{Quadrisection diagram for $S^1\times S^1\times S^3$} \label{figS1S1S3}
\end{center}
\end{figure}

In dimension $6$, we get infinitely many $6$--manifolds admitting a $5$--section of genus $4$, namely the $S^2$--bundles over~$S^4$. Such a bundle can be constructed by gluing two copies of $S^2\times B^4$ via a map $\varphi :S^3\to SO(3)$. We write $S^3$ as the quotient of $S^2\times [0,2\pi]$ by the shrinking of $S^2\times\{0\}$ and $S^2\times \{2\pi\}$; for $m\in\Z$, we define a map $\varphi_m$ that sends $(x,t)\in S^2\times [0,2\pi]$ onto the rotation of axis given by $x$ and angle $mt$. This defines an $S^2$--bundle $W(m)$. While $W(-m)$ is diffeomorphic to $W(m)$, the group $\pi_3\big(W(m)\big)$ is finite of order $|m|$, which shows that the $W(m)$ for $m\in\Z_{\geq0}$ are non-diffeomorphic (the order of $\pi_3\big(W(m)\big)$ can be computed from the homotopy exact sequence of the fibration). To get a simple multisection diagram of $W(m)$, it appears helpful to modify the map $\varphi_m$ by a homotopy. We write $S^3$ as $S^2\times [-1,2\pi+1]$ with $S^2\times\{-1\}$ and $S^2\times \{2\pi+1\}$ shrunk, and we define $\varphi_m$ as previously on $S^2\times[0,2\pi]$ and constant equal to the identity on $S^2\times([-1,0]\cup[2\pi,2\pi+1])$. We see $S^4$ as the union of two $4$--balls $B_a$ and $B_b$, and we further decompose $B_b$ into four $4$--balls as follows.
Write $B_b$ as the quotient of $B^3\times[-1,2\pi+1]$ by the shrinking of $B^3\times\{-1\}$ and $B^3\times\{2\pi+1\}$. Then, for $2\leq i\leq 5$, define $B_i$ as the image of $\Gamma_i\times[-1,2\pi+1]$ in this quotient, where the $\Gamma_i\subset B^3$ are the $3$--balls depicted in Figure~\ref{figDecB3}. Setting $B_1=B_a$, this provides a good ball decomposition $S^4=\cup_{1\leq i\leq 5}B_i$.

\begin{figure}[htb] 
\begin{center}
\begin{tikzpicture}
\begin{scope}
 \draw[fill=gray!20,opacity=0.8] (0,0) circle (2);
 \draw[dashed] (-2,0) .. controls +(1,0.8) and +(-1,0.8) .. (2,0);
 \draw[dashed,fill=gray!60,opacity=0.8] (0,-2) .. controls +(0.8,0.2) and +(0,-0.8) .. (1,-0.5) node {$\scriptstyle\bullet$} node[below right] {$r$} -- (-1,0.5) node {$\scriptstyle\bullet$} node[above right] {$p$} .. controls +(0,-0.8) and +(-0.8,0.4) .. (0,-2);
 \draw[dashed,fill=gray!40,opacity=0.6] (0,2) .. controls +(0.8,-0.2) and +(0,0.8) .. (1.2,0.4) node {$\scriptstyle\bullet$} node[above right] {$s$} -- (-1.2,-0.4) node {$\scriptstyle\bullet$} node[below] {$q$} .. controls +(0,1.2) and +(-0.8,-0.4) .. (0,2) node {$\scriptstyle\bullet$} node[above] {$n$};
 \draw (-2,0) .. controls +(1,-0.8) and +(-1,-0.8) .. (2,0);
 \draw (-2,1.6) node {$\Gamma_2$} (2,1.6) node {$\Gamma_3$} (-2,-1.6) node {$\Gamma_4$} (2,-1.6) node {$\Gamma_5$};
\end{scope}
\end{tikzpicture}
\caption{Decomposition of $B^3$ into four $3$--balls} \label{figDecB3}
\end{center}
\end{figure}

Now the bundle $W(m)$ is given by the gluing of $S^2\times B_a$ and $S^2\times B_b$ via the map $\varphi_m:\partial B_a=\partial B_b\to SO(3)$. We need to define the disk sections. For every $x\in S^2$, we fix a small neighborhood $\nu(x)$. We fix five points on $S^2$ represented on Figure~\ref{figDecB3}. In the fiber above any point of $B_1$, we choose the disk $D_1$ inside $\nu(n)$. Note that $\varphi_m(n)$ runs through all $S^2$, for $m>0$. If $z=(x,t)\in\Gamma_4\times[0,2\pi]\subset B_4$, then we choose the disk $D_4$ in the fiber over $z$ in $\Gamma_4\cap\nu(x)$. If $z\in\{0\}\times[-1,2\pi+1]$, then we choose $D_4$ in $\nu(q)$. Then from $(x,0)$ to $(x,-1)$, we move the disk $D_4$ along a geodesic arc from $x$ to $q$, and similarly from $(x,2\pi)$ to $(x,2\pi+1)$, and from $(x,t)\in S^2\times[-1,2\pi+1]$ to $(0,t)\in B^3\times[-1,2\pi+1]$. This defines the disk section $D_4$. Define $D_5$ similarly using $s$ instead of $q$. For $D_2$, when $z=(x,t)\in\Gamma_2\times[0,2\pi]\subset B_2$, we choose the disk $D_2$ in the fiber over $z$ in $-\Gamma_2\cap-\nu(x)$. Then we extend as for $D_4$ using the point $r$. Finally, define $D_3$ similarly as $D_2$ using the point $p$.
This gives explicit disk sections, providing thanks to Theorem~\ref{th:SurfaceBundle} a multisection of $W(m)$.
The central surface $\Sigma$ is given by the two fibers above the shrinkings of $S^2\times\{-1\}$ and $S^2\times\{2\pi+1\}$ joined by five tubes. A careful analysis of the $3$--dimensional handlebodies of the multisection provides the diagram in Figure~\ref{figS2bundleS4}, where $W_{1234}$ is represented in green, $W_{1235}$ in blue, $W_{1245}$ in red, $W_{1345}$ in orange, and $W_{2345}$ in purple. Let us have a closer look at $W_{1234}$. It contains the solid tube $D_5\times B_{1235}$, and $\partial D_5\times B_{1235}$ lies in~$\Sigma$. This gives the green curve entirely drawn on the left hand side of Figure~\ref{figS2bundleS4}. To find the other three curves, we choose arcs on one of the fibers defining $\Sigma$ (see Figure~\ref{figS2bundleS4} left). Then we construct a properly embedded disk in $W_{1234}$ by pushing these arcs to the other fiber defining~$\Sigma$, tracking the modification of the arcs imposed by the movement of the disk sections. 

\begin{figure}[htb] 
\begin{center}
\begin{tikzpicture} [scale=0.18]
\begin{scope}
 \draw (0,0) circle (20);
 \foreach \s/\t/\c in {1/1/vert,-1/1/red,-1/-1/blue,1/-1/orange} {
 \draw[xscale=\s,yscale=\t] (7,7) circle (2);
 \draw[thick,\c,xscale=\s,yscale=\t] (7,7) circle (2.5);}
 \draw[thick,purple] (7,5) -- (7,-5) (5,-7) -- (-5,-7) (-7,-5) -- (-7,5) (0,0) circle (19);
 \foreach \s/\t/\c/\cc/\ccc in {1/1/blue/red/orange,-1/1/orange/vert/blue,-1/-1/vert/red/orange,1/-1/red/vert/blue} {
 \draw[thick,\c,xscale=\s,yscale=\t] (8.4,8.4) -- (14,14.3);
 \draw[thick,\cc,xscale=\s,yscale=\t] (8.7,8.1) -- (14.3,14);
 \draw[thick,\ccc,rounded corners=5pt,xscale=\s,yscale=\t] (8.9,7.7) -- (14.6,13.7);}
\end{scope}
\begin{scope} [xshift=44cm]
 \draw (0,0) circle (20);
 \foreach \s in {-1,1}
 \foreach \t in {-1,1}
 \draw[xscale=\s,yscale=\t] (7,7) circle (2);
 \draw[thick,purple] (7,5) -- (7,-5) (5,-7) -- (-5,-7) (-7,-5) -- (-7,5);
 \foreach \s/\t/\c/\cc/\ccc in {1/1/blue/red/orange,-1/1/orange/vert/blue,-1/-1/vert/red/orange,1/-1/red/vert/blue} {
 \draw[thick,\c,rounded corners=2pt,xscale=\s,yscale=\t] (8.4,8.4) -- (8.7,8.8) arc (45:380:2.5) -- (14,14.3);
 \draw[thick,\cc,rounded corners=4pt,xscale=\s,yscale=\t] (8.7,8.1) -- (9,9) arc (45:135:2.9) -- (-9.2,-5) arc (135:315:3.5) -- (11,5.6) -- (12,11) -- (14.3,14);
 \draw[thick,\ccc,rounded corners=4pt,xscale=\s,yscale=\t] (8.9,7.7) -- (9.3,9.2) arc (45:135:3.3) -- (-9.5,-4.7) arc (135:315:3.9) -- (11.6,5.6) -- (12.6,11) -- (14.6,13.7);}
\end{scope}
\end{tikzpicture}
\caption{Multisection diagram for $\CP^3$\\ {\footnotesize The five boundary circles on both sides are identified by translation. The picture represents~$W(1)$, which is diffeomorphic to $\CP^3$. For a diagram of $W(m)$, the green, blue, orange, and red curves on the right hand side have to turn $m$ times.}} \label{figS2bundleS4}
\end{center}
\end{figure}

It is an interesting open question to ask whether one can deduce from the diagram that the $W(m)$ are non-diffeomorphic for distinct non-negative values of $m$.

We now treat surface bundles over $S^2\times S^1$. 

\begin{corollary}
 A surface bundle over $S^2\times S^1$ admits a quadrisection of genus $8g+9$, where $g$ is the genus of the fiber.
\end{corollary}
\begin{proof}
 We use the following good ball decomposition of $S^2\times S^1$. The factor $S^2$ is cut into two disks; each of these disks product $S^1$ is then cut into two balls, see Figure~\ref{figS2xS1GBD}. In this decomposition $S^2\times S^1=\cup_{1\leq i\leq 4}M_i$, each $M_i$ is a $3$--ball, each $M_{ij}$ is made of two $2$--disks, each $M_{ijk}$ is made of four intervals and $M_{1234}$ contains exactly eight points. In the associated quadrisection of a surface bundle over $S^2\times S^1$, the central surface is made of $8$ copies of the fiber joined by $16$ tubes.
\end{proof}

\begin{figure}[htb] 
\begin{center}
\begin{tikzpicture}
 \foreach \x/\y/\z/\t/\h/\b in {0/0/4/0/violet/blue,0/-1/4/0/violet/blue,1/0/4/0/violet/blue,1/-1/3/41.5/violet/blue,2/0/2/60/violet/blue,2/-1/1/75.5/violet/blue,3/0/0/90/violet/blue,3/-1/-1/104.5/violet/blue,4/0/-2/120/violet/blue,4/-1/-3/138.5/violet/blue,5/0/-4/180/violet/blue,5/-1/-4/180/violet/blue,6/0/-4/180/violet/blue,6/-1/-4/180/violet/blue,7/0/4/0/blue/violet,7/-1/3/41.5/blue/violet,8/0/2/60/blue/violet,8/-1/1/75.5/blue/violet,9/0/0/90/blue/violet,9/-1/-1/104.5/blue/violet,10/0/-2/120/blue/violet,10/-1/-3/138.5/blue/violet,11/0/-4/180/blue/violet,11/-1/-4/180/blue/violet} {
 \draw[fill=\h!70] (1.2*\x-0.6*\y,\y +\z/8) -- (1.2*\x-0.6*\y,\y+0.5) arc (90:90+\t:0.5) -- (1.2*\x-0.6*\y,\y +\z/8);
 \draw[fill=\b!70] (1.2*\x-0.6*\y,\y +\z/8) -- (1.2*\x-0.6*\y,\y-0.5) arc (270:90+\t:0.5) -- (1.2*\x-0.6*\y,\y +\z/8);}
 \foreach \x/\y/\z/\t/\h/\b in {0/0/0/90/green/red,0/-1/1/104.5/green/red,1/0/2/120/green/red,1/-1/3/138.5/green/red,2/0/4/180/green/red,2/-1/4/180/green/red,3/0/-4/0/red/green,3/-1/-4/0/red/green,4/0/-4/0/red/green,4/-1/-3/41.5/red/green,5/0/-2/60/red/green,5/-1/-1/75.5/red/green,6/0/0/90/red/green,6/-1/1/104.5/red/green,7/0/2/120/red/green,7/-1/3/138.5/red/green,8/0/4/180/red/green,8/-1/4/180/red/green,9/0/4/180/red/green,9/-1/4/180/red/green,10/0/-4/0/green/red,10/-1/-3/41.5/green/red,11/0/-2/60/green/red,11/-1/-1/75.5/green/red} {
 \draw[fill=\h!70] (1.2*\x-0.6*\y,\y +\z/8) -- (1.2*\x-0.6*\y,\y+0.5) arc (90:\t-90:0.5) -- (1.2*\x-0.6*\y,\y +\z/8);
 \draw[fill=\b!70] (1.2*\x-0.6*\y,\y +\z/8) -- (1.2*\x-0.6*\y,\y-0.5) arc (-90:-90+\t:0.5) -- (1.2*\x-0.6*\y,\y +\z/8);}
 \foreach \x in {0,...,11}
 \foreach \y in {0,-1} {
 \draw (1.2*\x-0.6*\y,\y) circle (0.5);
 \foreach \a in {-0.5,0.5} 
 \draw[->] (1.2*\x-0.6*\y +\a,\y-0.05) -- (1.2*\x-0.6*\y +\a,\y);
 \foreach \b in {-0.5,0.5} 
 \draw (1.2*\x-0.6*\y,\b +\y) node {$\scriptscriptstyle\bullet$};}
 \foreach \x in {0,2,...,22} 
 \draw (0.6*\x-0.45,0.5) node {$\scriptscriptstyle\x$};
 \foreach \x in {1,3,...,23} 
 \draw (0.6*\x-0.45,-1.5) node {$\scriptscriptstyle\x$};
\end{tikzpicture}
\caption{Good ball decomposition of $S^2\times S^1$: the $S^2$--slices\\{\footnotesize Here, $S^1$ is regarded as $[0,24]/(0=24)$. We represent $M_1$ in blue, $M_2$ in violet, $M_3$ in red and $M_4$ in green.}} \label{figS2xS1GBD}
\end{center}
\end{figure}

\begin{figure}[htb] 
\begin{center}
\begin{tikzpicture}
 \draw[rounded corners=5pt] (0,1.5) -- (3.9,1.5) -- (4.5,0.9) -- (4.5,-0.9) -- (3.9,-1.5) -- (-3.9,-1.5) -- (-4.5,-0.9) -- (-4.5,0.9) -- (-3.9,1.5) -- (0,1.5);
 \foreach \t in {0,180} {
 \draw[rounded corners=2pt,rotate=\t] (2.5,2.6) -- (2.5,3.9) -- (2.7,4) -- (4,2.7) -- (3.9,2.5) -- (-3.9,2.5) -- (-4,2.7) -- (-2.7,4) -- (-2.5,3.9) -- (-2.5,2.6) (2.5,1.4) -- (2.5,-1.4);
 \draw[rounded corners=5pt,rotate=\t](1.5,1.4) -- (1.5,-1.4) (1.5,2.6) -- (1.5,3.9) -- (0.9,4.5) -- (-0.9,4.5) -- (-1.5,3.9) -- (-1.5,2.6);}
 \foreach \t in {0,90,180,270} {
 \draw[rounded corners=4pt,rotate=\t] (4,4) -- (4.8,3.2) -- (5.2,3.2) .. controls +(0,1) and +(1,0) .. (3.2,5.2) -- (3.2,4.8) -- (4,4);
 \draw[rounded corners=5pt,rotate=\t] (5.5,0) -- (5.5,0.9) -- (6,2.7) .. controls +(0.5,2) and +(2,0.5) .. (2.7,6) -- (0.9,5.5) -- (0,5.5);}
 \foreach \t/\c in {0/green,90/red,180/blue,270/violet} {
 \draw[\c,rotate=\t] (3.35,3.35) .. controls +(-0.2,0.3) and +(-0.3,0.2) .. (4,4);
 \draw[\c,rotate=\t,dashed] (3.35,3.35) .. controls +(0.3,-0.2) and +(0.2,-0.3) .. (4,4);
 }
 \foreach \t/\c in {0/violet,90/green,180/red,270/blue} {
 \draw[\c,rotate=\t] (4.6,4.6) .. controls +(-0.2,0.3) and +(-0.3,0.2) .. (5.3,5.3);
 \draw[\c,rotate=\t,dashed] (4.6,4.6) .. controls +(0.3,-0.2) and +(0.2,-0.3) .. (5.3,5.3);
 }
 \foreach \x/\c in {-5.5/violet,-2.5/violet,1.5/red,4.5/red} {
 \draw[\c] (\x,0) .. controls +(0.3,0.3) and +(-0.3,0.3) .. (\x+1,0);
 \draw[\c,dashed] (\x,0) .. controls +(0.3,-0.3) and +(-0.3,-0.3) .. (\x+1,0);
 }
 \foreach \y/\c in {-5.5/green,-2.5/green,1.5/blue,4.5/blue} {
 \draw[\c] (0,\y) .. controls +(0.3,0.3) and +(0.3,-0.3) .. (0,\y+1);
 \draw[\c,dashed] (0,\y) .. controls +(-0.3,0.3) and +(-0.3,-0.3) .. (0,\y+1);
 }
 \draw[green,rounded corners=5pt] (-4.3,-1.1) -- (-4.6,-1) -- (-4.6,1) -- (-4,1.7) -- (4,1.7) -- (4.1,1.3) (5.8,2) -- (5.9,2.8) .. controls +(0.4,1.8) and +(1.8,0.4) .. (2.8,5.9) -- (1,5.25) -- (-0.4,5.5) (-0.4,4.5) -- (-1,4.75) -- (-2.1,4) .. controls +(-0.2,-0.2) and +(0,0.5) .. (-2.3,2.6) (-2.3,1.4) -- (-2.3,-1.4) (-2.3,-2.6) -- (-2.3,-4) -- (-2.8,-4.15) -- (-4.15,-2.8) -- (-3.98,-2.7)
 (-5.75,1.9) -- (4,1.9) -- (4.6,1) -- (4.75,-1) -- (4.2,-1.2) (3.95,-2.6) -- (5.5,-3) .. controls +(0.2,-1.4) and +(1.4,-0.2) .. (3,-5.5) -- (2.1,-4) -- (2.1,-2.6) (2.1,-1.4) -- (2.1,1.4) (2.1,2.6) -- (2.1,4) -- (2.6,3.95) (3.2,5) -- (1,4.9) -- (-1,5.4) -- (-2,5.8) (-3.2,4.9) -- (-3.1,4.7) -- (-4.5,3) -- (-5,3.2)
 (3.9,-3.9) -- (4.7,-3.1) -- (5.3,-3.1) .. controls +(0.1,-1.2) and +(1.2,-0.1) .. (3.1,-5.3) -- (3.1,-4.7) -- (3.9,-3.9)
 (-0.2,5.5) -- (1,5.4) -- (2.6,5.97) (3.3,5.2) -- (2.3,5.5) -- (1,5.1) -- (-0.6,5.5) (-0.6,4.5) -- (-1,4.6) -- (-2,4) -- (-1.5,3.8) (-2.5,3.8) -- (-2.2,4) -- (-1,4.9) -- (-0.2,4.5)
 (-1.2,-4.2) -- (-2.9,-4.3) -- (-4.3,-2.9) -- (-3.95,-2.6) (-4.2,-1.2) -- (-5.7,-2.9) .. controls +(-0.3,-1.6) and +(-1.6,-0.3) .. (-2.9,-5.7) -- (-1.6,-5.7);
 \draw[green,dashed] (4.1,1.3) -- (5.8,2) (-4.3,-1.1) -- (-3.98,-2.7)
 (4.2,-1.2) -- (3.95,-2.6) (2.6,3.95) -- (3.2,5) (-2,5.8) -- (-3.2,4.9) (-5,3.2) -- (-5.75,1.9)
 (2.6,5.97) -- (3.3,5.2) (-2.5,3.8) .. controls +(0.3,-0.3) and +(-0.3,-0.3) .. (-1.5,3.8)
 (-3.95,-2.6) -- (-4.2,-1.2) (-1.6,-5.7) -- (-1.2,-4.2);
 \foreach \x in {-0.2,-0.4,-0.6}
 \draw[green,dashed,xshift=\x cm] (0,4.5) .. controls +(-0.3,0.3) and +(-0.3,-0.3) .. (0,5.5);
 \draw[violet,rounded corners=5pt] 
 (-1,-4.4) -- (-3,-4.5) -- (-4.5,-3) -- (-4.2,-2.6) -- (-3.9,-2.55) (-4,-1.4) -- (-5.5,-3) .. controls +(-0.2,-1.4) and +(-1.4,-0.2) .. (-3,-5.5) -- (-1.25,-5.6)
 (3.7,4.3) -- (4.6,3.05) -- (5.1,1) -- (5.25,-1) -- (4,-1.75) -- (-4,-2.25) -- (-5.3,-3.1) .. controls +(-0.1,-1.2) and +(-1.2,-0.1) .. (-3.1,-5.3) -- (-1,-5.4) -- (-0.5,-5.5) (-0.5,-4.5) -- (1,-4.6) -- (1.7,-4) -- (1.7,-2.6) (1.7,-1.4) -- (1.7,1.4) (1.7,2.6) -- (1.7,4) -- (2.6,4.3) -- (3,3.7)
 (-5.85,2.1) -- (4,2.1) -- (4.75,1) -- (4.9,-1) -- (4,-1.4) (3.87,-2.51) -- (5.7,-2.9) .. controls +(0.3,-1.6) and +(1.6,-0.3) .. (2.9,-5.7) -- (1.9,-4) -- (1.9,-2.6) (1.9,-1.4) -- (1.9,1.4) (1.9,2.6) -- (1.9,4) -- (2.5,4.2) -- (2.7,3.98) (3.2,4.9) -- (1,4.75) -- (-1,5.25) -- (-2.35,5.9) (-3.2,5.1) -- (-3,4.5) -- (-4.3,2.9) -- (-4.8,3) -- (-5.1,3.2)
 (-3.52,3.55) -- (-4.15,2.8) -- (-4.6,2.8) -- (-5.3,3.1) .. controls +(-0.1,1.2) and +(-1.2,0.1) .. (-3.1,5.3) -- (-2.9,4.3) -- (-3.52,3.55)
 (4.4,3.6) -- (4.7,3.1) -- (5.25,1) -- (5.4,-1) -- (4,-1.9) -- (3.6,-2.5) (3.6,-1.5) -- (4,-1.6) -- (5.1,-1) -- (4.9,1) -- (4.5,3) -- (3.7,3);
 \draw[violet,dashed] (-3.9,-2.55) -- (-4,-1.4) (-1.25,-5.6) -- (-1,-4.4)
 (3,3.7) .. controls +(0.3,-0.2) and +(0.2,-0.3) .. (3.7,4.3) (-0.5,-5.5) .. controls +(-0.3,0.3) and +(-0.3,-0.3) .. (-0.5,-4.5)
 (4,-1.4) -- (3.87,-2.51) (2.7,3.98) -- (3.2,4.9) (-2.35,5.9) -- (-3.2,5.1) (-5.1,3.2) -- (-5.85,2.1)
 (3.7,3) .. controls +(0.3,-0.2) and +(0.2,-0.3) .. (4.4,3.6) (3.6,-2.5) .. controls +(-0.3,0.3) and +(-0.3,-0.3) .. (3.6,-1.5);
 \draw[blue,rounded corners=5pt] 
 (-5.72,1.7) -- (-5.4,1) -- (-5.4,-1) -- (-5.9,-2.8) .. controls +(-0.4,-1.8) and +(-1.8,-0.4) .. (-2.8,-5.9) -- (-2,-5.8) (-1.3,-4.1) -- (-1.7,-4) -- (-1.7,-2.6) (-1.7,-1.4) -- (-1.7,1.4) (-1.7,2.6) .. controls +(0,0.4) and +(-0.2,-0.2) .. (-1.5,3.2) (-2.5,3.2) -- (-2.3,4) -- (-2.8,4.15) -- (-3,3.7) (-3.7,4.3) -- (-4.65,3.1) -- (-4.8,3.25)
 (-0.5,-2.5) .. controls +(0.4,0.2) and +(-1,0) .. (4,-2.4) -- (4.3,-2.9) -- (2.9,-4.3) -- (1,-4.75) -- (-1,-4.6) -- (-1.9,-4) -- (-1.9,-2.6) (-1.9,-1.4) -- (-1.9,1.4) (-1.9,2.6) .. controls +(0,0.5) and +(-0.1,-0.2) .. (-1.6,4) -- (-1.3,4.1) (-1.6,5.69) -- (-2.8,5.85) .. controls +(-1.8,0.4) and +(-0.4,1.8) .. (-5.85,2.8) -- (-5.1,1) -- (-5.1,-1) -- (-4,-1.9) .. controls +(0.5,0) and +(-0.4,-0.2) .. (-0.5,-1.5)
 (-3.5,-1.5) -- (-4,-1.75) -- (-4.9,-1) -- (-4.9,1) -- (-5.8,2.85) -- (-5.2,3.3) (-6.03,2.8) -- (-5.9,2.75) -- (-5.25,1) -- (-5.25,-1) -- (-4,-2.1) -- (-3.5,-2.5)
 (2.5,3.8) -- (2.3,3.9) -- (2.3,2.6) (2.3,1.4) -- (2.3,-1.4) (2.3,-2.6) -- (2.3,-4) -- (2.8,-4.15) -- (4.15,-2.8) -- (4,-2.7) (4.3,-1.1) -- (4.6,-1) .. controls +(0,0.3) and +(0.1,-0.2) .. (4.5,0.5) (5.5,0.5) -- (5.4,1) -- (5.2,2.2) -- (5.35,3.1) .. controls +(0.1,1.2) and +(1.2,0.1) .. (3.1,5.35) -- (3.2,5.1)
 (5.55,1.1) -- (5.4,2.2) -- (5.5,3) .. controls +(0.2,1.4) and +(1.4,0.2) .. (3,5.5) -- (3.1,4.65) -- (4.3,2.9) -- (4.4,1)
 ;
 \draw[blue,dashed] 
 (-2,-5.8) -- (-1.3,-4.1) (-2.5,3.2) .. controls +(0.3,-0.3) and +(-0.3,-0.3) .. (-1.5,3.2) (-3,3.7) .. controls +(0.2,0.3) and +(0.3,0.2) .. (-3.7,4.3) (-4.8,3.25) -- (-5.72,1.7)
 (-0.5,-2.5) .. controls +(-0.3,0.3) and +(-0.3,-0.3) .. (-0.5,-1.5) (-1.3,4.1) -- (-1.6,5.69)
 (-3.5,-2.5) .. controls +(-0.3,0.3) and +(-0.3,-0.3) .. (-3.5,-1.5) (-5.2,3.3) -- (-6.03,2.8)
 (4,-2.7) .. controls +(0.4,0.2) and +(0,-0.5) .. (4.3,-1.1) (4.5,0.5) .. controls +(0.3,-0.3) and +(-0.3,-0.3) .. (5.5,0.5) (3.2,5.1) .. controls +(-0.3,0) and +(-0.3,0.3) .. (2.5,3.8)
 (4.4,1) -- (5.55,1.1)
 ;
 \draw[red,rounded corners=5pt] 
 (-1,-2.5) .. controls +(0.4,0.3) and +(-1,0) .. (4,-2.25) -- (4.5,-3) -- (3,-4.5) -- (1,-4.9) -- (-1,-4.75) -- (-2.1,-4) -- (-2.1,-2.6) (-2.1,-1.4) -- (-2.1,1.4) (-2.1,2.6) .. controls +(0,0.5) and +(-0.1,-0.2) .. (-1.8,4) -- (-1.1,4.3) (-1.3,5.6) -- (-2.9,5.7) .. controls +(-1.6,0.3) and +(-0.3,1.6) .. (-5.7,2.9) -- (-4.75,1) -- (-4.75,-1) -- (-4,-1.6) .. controls +(0.5,0) and +(-0.4,-0.2) .. (-1,-1.5)
 (-3.87,-3.87) -- (-4.65,-3.1) -- (-4.3,-2.5) -- (-4,-2.4) -- (0,-2.2) -- (4,-2.1) -- (5.9,-2.8) .. controls +(0.4,-1.8) and +(1.8,-0.4) .. (2.8,-5.85) -- (1,-5.25) -- (-1,-5.1) -- (-3.1,-4.65) -- (-3.87,-3.87)
 (0,2.3) -- (4,2.3) -- (4.1,2.8) -- (2.9,4.3) --(1,4.6) -- (-1,5.1) -- (-3,5.5) .. controls +(-1.4,0.2) and +(-0.2,1.4) .. (-5.5,3) -- (-4,2.3) -- (0,2.3)
 (5.67,1.5) -- (5.6,2.2) -- (5.7,2.9) .. controls +(0.3,1.6) and +(1.6,0.3) .. (2.9,5.7) -- (3,4.5) -- (4.2,2.8) -- (4.2,1.2)
 (-3.2,-5) -- (-1,-5.25) -- (1,-5.4) -- (3,-6.07) (3.2,-5.1) -- (2.5,-5.5) -- (1,-5.1) -- (-1,-4.9) -- (-2.3,-4.6) -- (-2.7,-3.98)
 ;
 \draw[red,dashed] 
 (-1,-2.5) .. controls +(-0.3,0.3) and +(-0.3,-0.3) .. (-1,-1.5) (-1.1,4.3) -- (-1.3,5.6)
 (4.2,1.2) -- (5.67,1.5)
 (3,-6.07) -- (3.2,-5.1) (-2.7,-3.98) -- (-3.2,-5)
 ;
 \draw[green,thick,->,opacity=0.2] (6.5,-5) -- (5.1,-4.3);
 \draw[green,thick,->,opacity=0.2] (6.7,5) -- (5.7,4.6);
 \draw[green,thick,->,opacity=0.2] (-6.8,-4.1) -- (-5.8,-3.8);
 \draw[green,thick,->,opacity=0.2] (6.7,-4) -- (5.55,-3.8);
 \draw[green,thick,->,opacity=0.2] (1,6) -- (1.2,5.5);
 \draw[green,thick,opacity=0.4] (6.5,-5.1) node[right] {$\partial\delta_1$}
 (6.7,5) node[right] {$\partial\delta_2$}
 (-6.8,-4.1) node[left] {$\partial\delta_3$}
 (6.7,-4) node[right] {$\partial\delta_4$}
 (1,6) node[above] {$\partial\delta_5$};
\end{tikzpicture}
\caption{Quadrisection diagram for $S^2\times S^2\times S^1$\\ {\footnotesize The cut system representing $W_{234}$, $W_{134}$, $W_{124}$ and $W_{123}$ is in blue, violet, red and green respectively.}} \label{figS2S2S1}
\end{center}
\end{figure}

In Figure~\ref{figS2S2S1}, we give a quadrisection diagram for $S^2\times S^2\times S^1$. The surface diagram is the central surface of the quadrisection, made of $8$ $2$--spheres, one above each point of $M_{1234}$, joined by $16$ tubes. The $8$ points of $M_{1234}$ appear on the slices $3$, $9$, $15$ and $21$ of Figure~\ref{figS2xS1GBD}, two points on each slice, and the corresponding $2$--spheres are drawn on Figure~\ref{figS2S2S1} respectively at the upper-right, upper-left, lower-left and lower-right corners. The tubes are given by the $\partial D_j\times M_{j,j+1,j+2}$, where $1\leq j\leq 4$ and each $M_{j,j+1,j+2}$ is made of $4$ arcs; the tubes corresponding to $j=1,2,3,4$ are circled by a meridian of color blue, violet, red and green  respectively on Figure~\ref{figS2S2S1}.

To draw the diagram curves on the \mbox{genus--$9$} surface, we need to determine, for each $3$--dimensional handlebody of the quadrisection, a family of $9$ properly embedded disks that cut the handlebody into a $3$--ball. We have 
$$W_{123}=\Big(M_{123}\times\big(S^2\setminus(D_1\cup D_2\cup D_3)\big)\Big)\cup\Big(M_{23}\times\partial D_2\Big)\cup\Big(M_{13}\times\partial D_3\Big)\cup\Big(M_{412}\times D_4\Big).$$
From Figure~\ref{figS2xS1GBD}, it can be read that $M_{23}$ (resp. $M_{13}$) is made of two disks, viewed as one bigon and one hexagon regarding the repartition of their boundaries along $M_{123}$ and $M_{234}$ (resp. $M_{123}$ and $M_{341}$). 
Our $9$ disks in $W_{123}$ are as follows (see Figure~\ref{figdisks} for the second and third points):
\begin{itemize} \itemsep=0cm
 \item $4$ meridional disks in the four solid tubes composing $D_4\times M_{412}$,
 \item a disk $\delta_1$ obtained as the union of:
 \begin{itemize}
  \item the product of a point in $\partial D_2$ and the bigon in $M_{23}$,
  \item the product of an arc joining $\partial D_1$ to $\partial D_2$ on $S^2$ and the relevant interval in $M_{123}$ (the one on the boundary of the bigon in $M_{23}$),
 \end{itemize}
 \item a disk $\delta_2$ obtained as the union of:
 \begin{itemize}
  \item the product of a point in $\partial D_2$ and the hexagon in $M_{23}$,
  \item the product of an arc joining $\partial D_1$ to $\partial D_2$ on $S^2$ and the relevant three intervals in~$M_{123}$,
 \end{itemize}
 \item disks $\delta_3$ and $\delta_4$ analogous to $\delta_1$ and $\delta_2$ with $\partial D_3$ instead of $\partial D_2$ and $M_{13}$ instead of $M_{23}$,
 \item a disk $\delta_5$ obtained as the product of an arc joining $\partial D_1$ to itself around $\partial D_2$ and an interval in $M_{123}$ (different from the ones corresponding to the bigons).
\end{itemize}
Figure~\ref{figS2xS2xS1} indicates which curves on the diagram represent the boundaries of the five disks $\delta_i$.

\begin{figure}[htb] 
\begin{center}
\begin{tikzpicture} 
\begin{scope}[scale=0.5]
 \foreach \s in {-1,1} 
 \draw[xscale=\s] (0,0) arc (90:70:10) arc (180:160:10) arc (70:90:11.8);
 \draw (-2.9,0) node {$S^2$};
 \foreach \s/\i in {-1/1,1/2}
 \draw[fill=white] (2*\s,1.5) circle (0.6) node {$\scriptstyle{D_\i}$};
 \draw (-1.5,1.8) node {$\scriptscriptstyle\bullet$} .. controls +(1,0.5) and +(-1,0.5) .. (1.5,1.8) node {$\scriptscriptstyle\bullet$}
 (0,2.1) node[above] {$\scriptstyle\gamma$}
 (-1.1,1.6) node {$\scriptstyle{*_1}$}
 (1.1,1.6) node {$\scriptstyle{*_2}$};
\end{scope}
\begin{scope}[xshift=4cm,scale=0.5]
 \draw[fill=green!40] (0,3) -- (0,0) .. controls +(0.6,0.8) and +(-1,-0.2) .. (4,2) -- (4,5) .. controls +(-1,-0.2) and +(0.6,0.8) .. (0,3);
 \draw[fill=green!60] (0,0) .. controls +(2,0) and +(4,0) .. (4,2) .. controls +(-1,-0.2) and +(0.6,0.8) .. (0,0);
 \draw[green,thick,opacity=0.2] (0,0) .. controls +(2,0) and +(4,0) .. (4,2) -- (4,5) .. controls +(-1,-0.2) and +(0.6,0.8) .. (0,3) -- (0,0);
 \draw[->] (1,-1) node[below] {{\small bigon in $M_{23}$}} -- (2.5,1);
 \draw[->] (5,0) node[below] {{\small interval in $M_{234}$}} -- (4,0.9);
 \draw[->] (5.5,4) node[above] {{\small interval}} (6.5,3.8) node {{\small in $M_{123}$}} (5,4) -- (3.5,1.9);
 \draw[<->] (-0.3,1.5) node[left] {$\scriptstyle{\gamma}$} (-0.3,0) node[left] {$\scriptstyle{*_2}$} -- (-0.3,3) node[left] {$\scriptstyle{*_1}$};
\end{scope}
\begin{scope}[xshift=9.6cm,yshift=0.4cm,scale=0.5]
 \draw[fill=green!40] (1.8,0.5) -- (1.8,3.5) -- (0,3) -- (0,0)
 (4.2,0.5) -- (4.2,3.5) -- (6,3) -- (6,0);
 \draw[fill=green!60] (0,0) -- (1.8,0.5) -- (4.2,0.5) -- (6,0) -- (4,-0.5) -- (2,-0.5) -- (0,0);
 \draw[green,thick,opacity=0.2] (0,0) -- (0,3) -- (1.8,3.5) -- (1.8,0.5) -- (4.2,0.5) -- (4.2,3.5) -- (6,3) -- (6,0) -- (4,-0.5) -- (4,2.5) -- (2,2.5) -- (2,-0.5) -- (0,0);
 \draw[->] (1,-2) node[below] {{\small hexagon in $M_{23}$}} -- (1.3,0);
 \draw[->] (4.5,-1.2) node[below] {{\small intervals in $M_{234}$}} -- (4.1,0.5);
 \draw[->] (4.7,-1.2) -- (5,-0.3);
 \draw[->] (4.3,-1.2) .. controls +(-0.6,0.3) and +(0,-1) .. (1.6,-0.4);
 \draw[->] (3,4) node[above] {{\small intervals in $M_{123}$}} -- (3,-0.5);
 \foreach \s in {-1,1}
 \draw[->] (3+0.5*\s,4) -- (3+2*\s,0.3);
 \draw[fill=green!40,opacity=0.6]
 (2,-0.5) -- (2,2.5) -- (4,2.5) -- (4,-0.5) ;
 \draw[<->] (-0.3,1.5) node[left] {$\scriptstyle{\gamma}$} (-0.3,0) node[left] {$\scriptstyle{*_2}$} -- (-0.3,3) node[left] {$\scriptstyle{*_1}$};
\end{scope}
\end{tikzpicture}
\caption{Meridional disks in $W_{123}$} \label{figdisks}
\end{center}
\end{figure}

The other cut-systems are obtained similarly.
We may note that a simpler quadrisection diagram of $S^2\times S^2\times S^1$ is provided in the next section.

We end this section with surface bundles over real projective spaces. 

\begin{corollary}
 A surface bundle over $\RP^n$ admits a multisection of genus $2^ng+2^{n-1}(n-1)+1$, where $g$ is the genus of the fiber.
\end{corollary}
\begin{proof}
 We need to give a suitable good ball decompositions of $\RP^n$. Using homogeneous coordinates $[x_0:\dots:x_n]$ for $\RP^n$, we set $M_i=\{|x_i|\geq|x_j|\mid \forall j\}$. This defines a good ball decomposition of $\RP^n$ where $M_{\{0,\dots,n\}}$ contains $2^n$ points and each $M_{\{0,\dots,n\}\setminus\{i\}}$ is made of $2^{n-1}$ intervals, for a total of $e=(n+1)2^{n-1}$ intervals. 
\end{proof}

\section{Multisecting fiber bundles over the circle}
\label{sec:bundlesS1}

For a closed $3$--manifold that fibers over the circle, there is a simple construction of a Heegaard splitting, which is the one described in the previous section. It has been generalized to dimension~$4$ by Gay and Kirby \cite{GayKirby} and Koenig \cite{Koenig}. A similar construction can indeed be performed in any dimension. 

\begin{theorem} \label{th:bundleS1}
 Let $W$ be an $(n+1)$--manifold which fibers over $S^1$ with fiber a closed $n$--manifold $X$ and monodromy $\varphi$. Assume that $X$ admits a multisection $X=\cup_{i=1}^{n-1}X_i$ preserved by $\varphi$, meaning that there is a permutation $\sigma$ of $\{1,\dots,n-1\}$ such that $\varphi(X_i)=X_{\sigma(i)}$. Then $W$ admits a multisection of genus $(c+n-1)g+1$, where $g$ is the genus of $S=X_{\{1,\dots,n-1\}}$ and $c$ is the number of cycles in the decomposition of~$\sigma$ (including fixed points as order--$1$ cycles). 
\end{theorem}

The following proof generalizes Koenig's proof in dimension $4$. In that case, the permutation $\sigma$ is the identity or a transposition; Koenig indeed obtains a trisection of genus $4g+1$ in the first case and $3g+1$ in the second case. Further, he shows that, in the case where $\sigma$ is the identity, some destabilizations can be performed such that one finally ends with a trisection of genus $3g+1$ in both cases. Similar destabilizations can be performed in some examples in higher dimensions, as in Figures~\ref{figCP2xS1} and~\ref{figS2xS2xS1}, but we could not find a general argument. It would be interesting to determine whether one can always get a multisection of genus $ng+1$.

\begin{proof}
We construct a multisection of $W$ using the multisection of $X$ and a decomposition of $S^1$ into intervals. Identifying $W$ with $\fract{X\times I}{(x,0)\sim(\varphi(x),1)}$, we define the $W_i$ in the \mbox{product $X\times I$.}

\begin{figure}[htb] 
\begin{center}
\begin{tikzpicture} [xscale=0.7,yscale=0.7]
 \foreach \x in {0,5} \draw (\x,0) -- (\x,3);
 \foreach \y in {0,1,2,3} \draw (0,\y) -- (5,\y);
 \foreach \x/\y in {1/0,2/2,3/1,4/0} \draw (\x,\y) -- (\x,\y+1);
 \foreach \x/\y/\i in {0.5/0/3,2.5/0/4,4.5/0/1,1.5/1/1,4/1/2,1/2/2,3.5/2/3} \draw (\x,\y+0.5) node {$W_\i$};
 \foreach \y/\i in {0/1,1/2,2/3} \draw (-0.7,\y+0.5) node {$X_{\i}$};
 \draw (2.5,-0.8) node {$\sigma=(123)$};
\begin{scope} [xshift=6cm]
 \foreach \x in {0,6} \draw (\x,0) -- (\x,3);
 \foreach \y in {0,1,2,3} \draw (0,\y) -- (6,\y);
 \foreach \x/\y in {1/1,2/2,3/1,4/0,5/0} \draw (\x,\y) -- (\x,\y+1);
 \foreach \x/\y/\i in {2/0/1,4.5/0/4,5.5/0/1,0.5/1/3,2/1/4,4.5/1/2,1/2/2,4/2/3} \draw (\x,\y+0.5) node {$W_\i$};
 \draw (3,-0.8) node {$\sigma=(23)$};
\end{scope}
\begin{scope} [xshift=13cm]
 \foreach \x in {0,7} \draw (\x,0) -- (\x,3);
 \foreach \y in {0,1,2,3} \draw (0,\y) -- (7,\y);
 \foreach \x/\y in {1/2,2/2,3/1,4/1,5/0,6/0} \draw (\x,\y) -- (\x,\y+1);
 \foreach \x/\y/\i in {2.5/0/1,5.5/0/4,6.5/0/1,1.5/1/2,3.5/1/4,5.5/1/2,0.5/2/3,1.5/2/4,4.5/2/3} \draw (\x,\y+0.5) node {$W_\i$};
 \draw (2.5,-0.8) node {$\sigma=id$};
\end{scope}
\end{tikzpicture}
\caption{Scheme of the quadrisection of $W$ for different permutations $\sigma$} \label{figbundleS1n4}
\end{center}
\end{figure}

Assume first $n=4$. There are three cases, depending on the type of $\sigma$: identity, transposition or $3$--cycle. These are schematized in Figure~\ref{figbundleS1n4}. For instance, when $\sigma=(123)$, we first set:
\begin{align*}
 W_1'&=\left(X_1\times\left[\frac45,1\right]\right)\cup_\varphi \left(X_2\times\left[0,\frac35\right]\right),&
 W_2'&=\left(X_2\times\left[\frac35,1\right]\right)\cup_\varphi \left(X_3\times\left[0,\frac25\right]\right),\\
 W_3'&=\left(X_3\times\left[\frac25,1\right]\right)\cup_\varphi \left(X_1\times\left[0,\frac15\right]\right),&
 W_4'&=X_1\times\left[\frac15,\frac45\right].
\end{align*}

At this stage, the $W_i'$ are handlebodies, but their intersections are not. We shall again fix this by tubing. We say that a $4$--ball in $X$ is in {\em good position} if it is transverse to the $X_i$, the $X_{ij}$ and the central surface $S$, and if it intersects each of these pieces along a ball of the corresponding dimension. Note that such balls do exist: take a small enough neighborhood of a point in the central surface. For each interval in our construction, namely $\left[\frac i5,\frac {i+1}5\right]$ for $i=1,2,3$ and $\left[\frac45,1\right]\cup\left[0,\frac15\right]$, we take a tube $B^4\times I$ made of a $4$--ball in good position above each point of the interval, and we add it to the $W_i'$ which doesn't appear above this interval. For instance, a tube above $\left[\frac25,\frac35\right]$ is removed from $W_1\cup W_3\cup W_4$ and added to $W_2'$ to form $W_2$. We require distinct tubes to be disjoint.

Let us check that this defines a quadrisection of $W$. Each $W_i$ is a $4$--dimensional handlebody times an interval (namely $W_i'$) connected to itself by a $1$--handle, thus a $5$--dimensional handlebody. The $W_{ij}$ are made of copies of some $X_k$ and copies of some $X_{k\ell}$ times interval, joined by two tubes, hence they are $4$--dimensional handlebodies; for instance, $W_{12}$ is the union of $X_2\times\left\lbrace\frac35\right\rbrace$ and $\left(X_{12}\times\left[\frac45,1\right]\right)\cup_{\varphi}\left(X_{23}\times\left[0,\frac25\right]\right)$ joined by the tubes above $\left[\frac25,\frac35\right]$ and $\left[\frac35,\frac45\right]$. The $W_{ijk}$ are made of two copies of some $X_{\ell m}$ and a copy of the surface $S=X_{123}$ times interval, joined by three tubes. Finally, the central surface $\Sigma=W_{1234}$ is given by four copies of $S$, above $\frac i5$ for $i=1,..,4$, joined by tubes. 

The genera of the surfaces are related by $g(\Sigma)=4g+1$, where $g=g(S)$. 
The schemes of Figure~\ref{figbundleS1n4} show how to apply this construction for other permutations of the trisection of $X$. The genus of $\Sigma$ is then given by $5g+1$ and $6g+1$ respectively. 

\begin{figure}[htb] 
\begin{center}
\begin{tikzpicture} [yscale=0.6]
 \foreach \x in {0,7} \draw (\x,0) -- (\x,5);
 \foreach \y in {0,1,2,4,5} \draw (0,\y) -- (7,\y);
 \foreach \x/\y in {1/0,6/0,5/1,2/4} \draw (\x,\y) -- (\x,\y+1);
 \foreach \x/\y/\i in {0.5/0/n-1,3.5/0/n,6.5/0/1,2.5/1/1,6/1/2,1/4/n-2,4.5/4/n-1} \draw (\x,\y+0.5) node {$W_{\i}$};
 \foreach \x in {1,3.5,6} \draw (\x,3) node {$\vdots$};
 \foreach \y/\i in {0.5/1,1.5/2,4.5/n-1} \draw (-0.7,\y) node {$X_{\i}$};
\end{tikzpicture}
\caption{Scheme of the multisection of $W$ for a monodromy inducing a cycle} \label{figbundleS1cycle}
\end{center}
\end{figure}

\begin{figure}[htb] 
\begin{center}
\begin{tikzpicture} [xscale=0.7,yscale=0.6]
 \foreach \x in {0,8} \draw (\x,0) -- (\x,5);
 \foreach \y in {0,1,...,5} \draw (0,\y) -- (8,\y);
 \foreach \x/\y in {5/0,7/0,6/1,1/2,4/2,3/3,2/4} \draw (\x,\y) -- (\x,\y+1);
 \foreach \x/\y/\i in {2.5/0/2,6/0/6,7.5/0/1,3/1/1,7/1/2,0.5/2/5,2.5/2/6,6/2/3,1.5/3/3,5.5/3/4,1/4/4,5/4/5} \draw (\x,\y+0.5) node {$W_\i$};
 \foreach \i in {1,2,...,5} \draw (-0.7,\i-0.5) node {$X_{\i}$};
 \draw (4,-0.8) node {$n=6\qquad\sigma=(12)(345)$};
\begin{scope} [xshift=11cm]
 \foreach \x in {0,10} \draw (\x,0) -- (\x,6);
 \foreach \y in {0,1,...,6} \draw (0,\y) -- (10,\y);
 \foreach \x/\y in {7/0,9/0,8/1,4/2,6/2,5/3,1/4,3/4,2/5} \draw (\x,\y) -- (\x,\y+1);
 \foreach \x/\y/\i in {3.5/0/2,8/0/7,9.5/0/1,4/1/1,9/1/2,2/2/4,5/2/7,8/2/3,2.5/3/3,7.5/3/4,0.5/4/6,2/4/7,6.5/4/5,1/5/5,6/5/6} \draw (\x,\y+0.5) node {$W_\i$};
 \foreach \i in {1,2,...,6} \draw (-0.7,\i-0.5) node {$X_{\i}$};
 \draw (5,-0.8) node {$n=7\qquad\sigma=(12)(34)(56)$};
\end{scope}
\end{tikzpicture}
\caption{Other schemes of multisections} \label{figbundleS1n67}
\end{center}
\end{figure}

\begin{figure}[htb] 
\begin{center}
\begin{tikzpicture} [xscale=0.6,yscale=0.5]
 \foreach \x in {0,22} \draw (\x,0) -- (\x,18);
 \foreach \y in {0,1,2,4,5,6,7,9,10,13,14,15,17,18} \draw (0,\y) -- (22,\y);
 \foreach \x/\y in {1/13,2/17,5/14,6/13,10/5,11/9,14/6,15/5,16/0,17/4,20/1,21/0} \draw (\x,\y) -- (\x,\y+1);
 \foreach \x/\y/\i in {8/0/i_1,18.5/0/n,21.5/0/1,10/1/1,21/1/2,8.5/4/i_1-1,19.5/4/i_1,5/5/i_2,12.5/5/n,18.5/5/i_1+1,7/6/i_1+1,18/6/i_1+2,5.5/9/i_2-1,16.5/9/i_2,0.5/13/i_c,3.5/13/n,14/13/i_{c-1}+1,2.5/14/i_{c-1}+1,13.5/14/i_{c-1}+2,1/17/i_c-1,12/17/i_c} \draw (\x,\y+0.5) node {$W_{\i}$};
 \foreach \x in {2,5,...,20} \foreach \y in {3.1,8.1,11.6,16.1} \draw (\x,\y) node {$\vdots$};
 \foreach \y/\i in {0.5/1,1.5/2,17.5/n-1} \draw (-0.8,\y) node {$X_{\i}$};
 \foreach \x/\y/\c in {15.5/0/blue,9/5/vert,0/13/red}
 \draw[line width=4pt,\c,opacity=0.4]
 (\x,\y) rectangle (\x+6.5,\y+5);
\end{tikzpicture}
\caption{Scheme for the permutation $(1\ \cdots\ i_1)(i_1+1\ \cdots\ i_2)\dots(i_{c-1}+1\ \cdots\ i_c)$\\ [2pt] \footnotesize{The blocks in blue, green and red correspond to the cycles $(1\ \cdots\ i_1)$, $(i_1+1\ \cdots\ i_2)$, and $(i_{c-1}+1\ \cdots\ i_c)$ respectively}} \label{figbundleS1general}
\end{center}
\end{figure}

In higher dimensions, the same construction works, as long as we can determine the right scheme. The idea is to use the decomposition of $\sigma$ into disjoint cycles. Figure~\ref{figbundleS1cycle} gives a model for an $(n-1)$--cycle. Then different cycles can be stacked together as examplified in Figures~\ref{figbundleS1n4} and~\ref{figbundleS1n67}; the general pattern is given in Figure~\ref{figbundleS1general}. In doing so, we decompose the interval $[0,1]$ into $N$ intervals $I_\ell=[\frac{\ell}{N+1},\frac{\ell+1}{N+1}]$ for $\ell=1,\dots,N-1$ and $I_N=[\frac{N}{N+1},1]\cup[0,\frac{1}{N+1}]$; note that $N$ equals the number of cycles in the decomposition of $\sigma$ (including fixed points as order--$1$ cycles) plus $n-1$. Also note that, above each interval $I_\ell$, $n-1$ distinct $W_i$'s appear (see Figure~\ref{figbundleS1general}). As above, we first define $W_i'$ as the disjoint union of the $X_k\times I_\ell$ for each interval $I_\ell$ above which $W_i$ appears in the $X_k$--line. We say that an $n$--ball in $X$ is in {\em good position} if it meets each $X_I$ transversely along a ball. 
Then, for each interval $I_\ell$, we take a tube $B^n\times I_\ell$ made of a ball in good position above each point of $I_\ell$, so that the different tubes are disjoint, and we add this tube to the only $W_i$ which doesn't appear above $I_\ell$ . This tube is thus removed from the other $W_j$'s; since the trace of the tube on a $W_j$ is a collar neighborhood of a disk in $\partial W_j$, this amounts to modifying $W_j$ by an isotopy. Thus constructed, the $W_i$ are made of $n$--dimensional handlebodies times interval joined by $1$--handles, hence they are $(n+1)$--dimensional handlebodies. For a non-empty proper subset $I$ of $\{1,\dots,n\}$, $W_I$ is made of some $X_J\times I_\ell$ with $|J|=|I|$ and some $X_J\times\{\frac{\ell}{N+1}\}$ with $|J|=|I|-1$, joined by $1$--handles; hence $W_I$ is an $(n-|I|+2)$--dimensional handlebody. 
Finally, $\Sigma=W_{\{1,\dots,n\}}$ is made of $N$ copies of the punctured $S$ joined by $N$ tubes.
\end{proof}

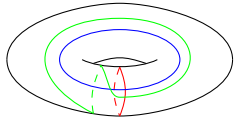
\begin{figure}[htb] 
\begin{center}
\begin{tikzpicture} 
\begin{scope} [scale=0.5,yshift=2cm]
 \draw (0,0) .. controls +(0,1) and +(-1,0) .. (3,1.5) .. controls +(1,0) and +(0,1) .. (6,0);
 \draw (0,0) .. controls +(0,-1) and +(-1,0) .. (3,-1.5) .. controls +(1,0) and +(0,-1) .. (6,0);
 \draw (2,0) ..controls +(0.5,-0.25) and +(-0.5,-0.25) .. (4,0);
 \draw (2.3,-0.1) ..controls +(0.6,0.2) and +(-0.6,0.2) .. (3.7,-0.1);
 \draw[red] (3,-0.2) ..controls +(0.2,-0.5) and +(0.2,0.5) .. (3,-1.5);
 \draw[dashed,red] (3,-0.2) ..controls +(-0.2,-0.5) and +(-0.2,0.5) .. (3,-1.5);
 \draw[blue] (3,0)ellipse(1.6 and 0.8);
 \draw[green] (2.5,-0.15) .. controls +(0.2,-0.3) and +(-0.3,0) .. (3,-1) .. controls +(2.5,0) and +(2.5,0) .. (3,1.1) .. controls +(-2.3,0) and +(-2,1) .. (2.3,-1.43);
 \draw[green,dashed] (2.3,-1.43) .. controls +(-0.1,0.5) and +(-0.2,-0.5) .. (2.5,-0.15);
\end{scope}
\begin{scope} [xscale=0.7,yscale=0.7,xshift=10cm]
 \foreach \x in {0,7} \draw (\x,0) -- (\x,3);
 \foreach \y in {0,1,2,3} \draw (0,\y) -- (7,\y);
 \foreach \y/\i in {0/1,1/2,2/3} \draw (-0.7,\y+0.5) node {$X_{\i}$};
 \foreach \x/\y in {1/0,2/1,3/2,4/2,5/1,6/0} \draw (\x,\y) -- (\x,\y+1);
 \foreach \x/\y/\i in {0.5/0/2,3.5/0/1,6.5/0/2,1/1/3,3.5/1/2,6/1/3,1.5/2/4,3.5/2/3,5.5/2/4} \draw (\x,\y+0.5) node {$W_\i$};
\end{scope}
\end{tikzpicture}
\caption{Trisection diagram for $\CP^2$ and scheme of a quadrisection of $\CP^2\times S^1$} \label{figCP2S1scheme}
\end{center}
\end{figure}

\begin{figure}[htb] 
\begin{center}
\begin{tikzpicture} 
\begin{scope} [scale=0.6]
 \draw (0,0) circle (6);
 \draw (-1,0) ..controls +(0.5,-0.4) and +(-0.5,-0.4) .. (1,0) (-0.8,-0.1) ..controls +(0.6,0.3) and +(-0.6,0.3) .. (0.8,-0.1);
 \foreach \t in {30,90,150,210,270,330}
 \draw[rotate=\t] (-1,-4) ..controls +(0.5,-0.4) and +(-0.5,-0.4) .. (1,-4) (-0.8,-4.1) ..controls +(0.6,0.3) and +(-0.6,0.3) .. (0.8,-4.1);
 \foreach \x/\c in {-0.1/orange,0.1/blue} {
 \draw[color=\c] (\x,-0.3) .. controls +(-0.3,-1) and +(-0.3,1) .. (\x,-6);
 \draw[color=\c,dashed] (\x,-0.3) .. controls +(0.3,-1) and +(0.3,1) .. (\x,-6);}
 \foreach \x/\c in {-0.1/purple,0.1/green} {
 \draw[color=\c] (\x,0.12) .. controls +(0.3,1) and +(0.3,-1) .. (\x,6);
 \draw[color=\c,dashed] (\x,0.12) .. controls +(-0.3,1) and +(-0.3,-1) .. (\x,6);} 
 \foreach \t in {90,150,210,270}
 \draw[rotate=\t,color=green] (0,-4.1) ellipse (1.2 and 0.7);
 \foreach \t/\c in {0/green,60/blue,120/purple,180/orange,240/purple,300/blue} {
 \draw[rotate=\t,color=\c,dashed] (-1.5,-3.75) .. controls +(0.6,-0.3) and +(-0.6,-0.3) .. (1.5,-3.75);
 \draw[rotate=\t,color=\c] (-1.5,-3.75) .. controls +(0.6,0.3) and +(-0.6,0.3) .. (1.5,-3.75);
 \draw[rotate=\t,color=\c] (0,-3.3) ellipse (3.3 and 1.7);}
 \foreach \t in {0,60,120,180} {
 \draw[rotate=\t,color=orange] (-6,0) .. controls +(0.4,-0.2) and +(-0.4,-0.2) .. (-4.3,0);
 \draw[rotate=\t,color=orange,dashed] (-6,0) .. controls +(0.4,0.2) and +(-0.4,0.2) .. (-4.3,0);}
 \foreach \t/\c/\s in {-60/purple/1,-120/purple/-1,60/blue/1,120/blue/-1} {
 \draw[rotate=\t,color=\c] (4.23,-0.5*\s) .. controls +(0.4,-0.4*\s) and +(0.4,0) .. (4,-1.27*\s) .. controls +(-0.4,0) and +(0,-0.6*\s) .. (3.3,0) .. controls +(0,0.6*\s) and +(-0.4,0) .. (4,1.27*\s) .. controls +(1.2,0) and +(-0.6,0.5*\s) .. (5.98,-0.5*\s);
 \draw[rotate=\t,color=\c,dashed] (4.23,-0.5*\s) .. controls +(0.5,0.2*\s) and +(-0.4,0.1*\s) .. (5.98,-0.5*\s);}
 \foreach \t/\i in {30/6,90/5,150/4,210/3,270/2,330/1} 
 \draw[rotate=\t] (0,-6.5) node {$S_\i$};
\end{scope}
\begin{scope} [scale=0.55,xshift=14.5cm,yshift=0.5cm]
 \draw (0,0) circle (6);
 \draw (-1,0) ..controls +(0.5,-0.4) and +(-0.5,-0.4) .. (1,0) (-0.8,-0.1) ..controls +(0.6,0.3) and +(-0.6,0.3) .. (0.8,-0.1);
 \foreach \t in {45,135,225,315}
 \draw[rotate=\t] (-1,-4) ..controls +(0.5,-0.4) and +(-0.5,-0.4) .. (1,-4) (-0.8,-4.1) ..controls +(0.6,0.3) and +(-0.6,0.3) .. (0.8,-4.1);
 \foreach \x/\c in {-0.1/orange,0.1/blue} {
 \draw[color=\c] (\x,-0.3) .. controls +(-0.3,-1) and +(-0.3,1) .. (\x,-6);
 \draw[color=\c,dashed] (\x,-0.3) .. controls +(0.3,-1) and +(0.3,1) .. (\x,-6);}
 \foreach \x/\c in {-0.1/purple,0.1/green} {
 \draw[color=\c] (\x,0.12) .. controls +(0.3,1) and +(0.3,-1) .. (\x,6);
 \draw[color=\c,dashed] (\x,0.12) .. controls +(-0.3,1) and +(-0.3,-1) .. (\x,6);} 
 \foreach \t/\c in {45/blue,135/green,225/green,315/blue}
 \draw[rotate=\t,color=\c] (0,-4.1) ellipse (1.2 and 0.7);
 \foreach \t/\c in {0/green,180/orange}
 \draw[rotate=\t,color=\c] (0,-2.8) ellipse (4.3 and 1.9);
 \foreach \t/\c in {0/green,90/purple,180/orange,270/purple} {
 \draw[rotate=\t,color=\c,dashed] (-2.4,-3.3) .. controls +(0.8,-0.4) and +(-0.8,-0.4) .. (2.4,-3.3);
 \draw[rotate=\t,color=\c] (-2.4,-3.3) .. controls +(0.8,0.4) and +(-0.8,0.4) .. (2.4,-3.3);}
 \foreach \t in {45,135} {
 \draw[rotate=\t,color=orange] (-6,0) .. controls +(0.4,-0.2) and +(-0.4,-0.2) .. (-4.3,0);
 \draw[rotate=\t,color=orange,dashed] (-6,0) .. controls +(0.4,0.2) and +(-0.4,0.2) .. (-4.3,0);}
 \foreach \t/\s in {45/1,135/-1} {
 \draw[rotate=\t,color=blue] (4.23,-0.5*\s) .. controls +(0.4,-0.4*\s) and +(0.4,0) .. (4,-1.27*\s) .. controls +(-0.4,0) and +(0,-0.6*\s) .. (3.3,0) .. controls +(0,0.6*\s) and +(-0.4,0) .. (4,1.27*\s) .. controls +(1.2,0) and +(-0.6,0.5*\s) .. (5.98,-0.5*\s);
 \draw[rotate=\t,color=blue,dashed] (4.23,-0.5*\s) .. controls +(0.5,0.2*\s) and +(-0.4,0.1*\s) .. (5.98,-0.5*\s);}
 \foreach \t/\s in {-45/1,-135/-1} {
 \draw[rotate=\t,color=purple] (4.23,-0.5*\s) .. controls +(0.4,-0.4*\s) and +(0.4,0) .. (4,-1.27*\s) .. controls +(-0.4,0) and +(1,-1*\s) .. (1,1*\s) .. controls +(-1,1*\s) and +(1,-1*\s) .. (-1,3*\s) .. controls +(-0.7,0.7*\s) and +(-0.7,-0.7*\s) .. (-1,5.2*\s) .. controls +(0.9,0.9*\s) and +(-1.5,1.5*\s) .. (3.5,3.5*\s) .. controls +(1.2,-1.2*\s) and +(-0.5,1*\s) .. (5.98,-0.5*\s);
 \draw[rotate=\t,color=purple,dashed] (4.23,-0.5*\s) .. controls +(0.5,0.2*\s) and +(-0.4,0.1*\s) .. (5.98,-0.5*\s);} 
 \draw (0,-6.8) node {\footnotesize{After destabilizing}};
\end{scope}
\end{tikzpicture}
\end{center}
\caption{Quadrisection diagrams of $\CP^2\times S^1$}
\label{figCP2xS1}
\end{figure}

\begin{figure}[hbt] 
\begin{center}
\begin{tikzpicture} 
\begin{scope}[xshift=1.4cm,yshift=12cm,scale=0.35]
\draw (0,0) ..controls +(0,1) and +(-2,1) .. (4,2) ..controls +(2,-1) and +(-2,-1) .. (8,2) ..controls +(2,1) and +(0,1) .. (12,0);
\draw (0,0) ..controls +(0,-1) and +(-2,-1) .. (4,-2) ..controls +(2,1) and +(-2,1) .. (8,-2) ..controls +(2,-1) and +(0,-1) .. (12,0);
\draw (2,0) ..controls +(0.5,-0.25) and +(-0.5,-0.25) .. (4,0);
\draw (2.3,-0.1) ..controls +(0.6,0.2) and +(-0.6,0.2) .. (3.7,-0.1);
\draw (8,0) ..controls +(0.5,-0.25) and +(-0.5,-0.25) .. (10,0);
\draw (8.3,-0.1) ..controls +(0.6,0.2) and +(-0.6,0.2) .. (9.7,-0.1);
\draw[color=green,dashed] (3.7,-0.1) .. controls +(1,-0.5) and +(-1,-0.5) .. (8.3,-0.1);
\foreach \x/\c in {3/red,9/blue} {
\draw[color=\c] (\x,-0.2) ..controls +(0.2,-0.5) and +(0.2,0.5) .. (\x,-2.3);
\draw[dashed,color=\c] (\x,-0.2) ..controls +(-0.2,-0.5) and +(-0.2,0.5) .. (\x,-2.3);
\draw[color=\c] (12-\x,0)ellipse(1.6 and 0.8);}
\draw[color=green] (1,0) .. controls +(0,-1) and +(-1,0) .. (3,-1.2) .. controls +(1,0) and +(-1,0) .. (6,-0.8) .. controls +(1,0) and +(-1,0) .. (9,-1.2) .. controls +(1,0) and +(0,-1) .. (11,0) .. controls +(0,1) and +(1,0) .. (9,1.2) .. controls +(-1,0) and +(1,0) .. (6,0.8) .. controls +(-1,0) and +(1,0) .. (3,1.2) .. controls +(-1,0) and +(0,1) .. (1,0);
\draw[color=green] (3.7,-0.1) .. controls +(1,0.5) and +(-1,0.5) .. (8.3,-0.1);
\end{scope}
\begin{scope} [scale=0.65,xshift=-4cm,yshift=13.5cm]
 \draw (0,0) circle (7);
 \draw (-1,0) ..controls +(0.5,-0.4) and +(-0.5,-0.4) .. (1,0) (-0.8,-0.1) ..controls +(0.6,0.3) and +(-0.6,0.3) .. (0.8,-0.1);
 \foreach \t in {30,90,150,210,270,330} {
 \draw[rotate=\t] (-1,-5.3) ..controls +(0.5,-0.4) and +(-0.5,-0.4) .. (1,-5.3) (-0.8,-5.4) ..controls +(0.6,0.3) and +(-0.6,0.3) .. (0.8,-5.4);
 \draw[rotate=\t] (-1,-3) ..controls +(0.5,-0.4) and +(-0.5,-0.4) .. (1,-3) (-0.8,-3.1) ..controls +(0.6,0.3) and +(-0.6,0.3) .. (0.8,-3.1);}
 \foreach \x/\c in {-0.1/orange,0.1/purple} {
 \draw[color=\c] (\x,-0.3) .. controls +(-0.3,-1) and +(-0.3,1) .. (\x,-7);
 \draw[color=\c,dashed] (\x,-0.3) .. controls +(0.3,-1) and +(0.3,1) .. (\x,-7);}
 \foreach \x/\c in {-0.1/blue,0.1/green} {
 \draw[color=\c] (\x,0.12) .. controls +(0.3,1) and +(0.3,-1) .. (\x,7);
 \draw[color=\c,dashed] (\x,0.12) .. controls +(-0.3,1) and +(-0.3,-1) .. (\x,7);} 
 \foreach \t in {90,150,210,270}
 \draw[rotate=\t,color=green] (0,-5.4) ellipse (1.2 and 0.6);
 \foreach \t in {30,90,270,330}
 \draw[rotate=\t,color=orange] (0,-3.1) ellipse (1.1 and 0.4);
 \foreach \s in {-1,1} {
 \draw[xscale=\s,color=green] (-2.85,0) .. controls +(0.4,-0.2) and +(-0.4,-0.2) .. (-0.8,-0.1);
 \draw[xscale=\s,color=green,dashed] (-2.85,0) .. controls +(0.4,0.2) and +(-0.4,0.2) .. (-0.8,-0.1);
 \draw[xscale=\s,color=green] (0.4,0.04) .. controls +(0.5,0.5) and +(-0,-0.6) .. (1.4,2.5);
 \draw[xscale=\s,color=green,dashed] (0.4,0.04) .. controls +(0,0.5) and +(-0.4,-0.2) .. (1.4,2.5);}
 \foreach \t/\c in {0/green,60/blue,120/purple,180/orange,240/purple,300/blue} {
 \draw[rotate=\t,color=\c,dashed] (-2,-5.05) .. controls +(0.6,-0.3) and +(-0.6,-0.3) .. (2,-5.05);
 \draw[rotate=\t,color=\c] (-2,-5.05) .. controls +(0.6,0.3) and +(-0.6,0.3) .. (2,-5.05);
 \foreach \s in {-1,1} 
 \draw[rotate=\t,color=\c,xscale=\s] (0,-6.4) .. controls +(2,0) and +(0.2,-1) .. (4,-4.3) .. controls +(-0.1,0.5) and +(0.8,0.2) .. (2.7,-3.8) .. controls +(-1.2,-0.3) and +(1,0) .. (0,-4.5);
 \draw[rotate=\t,color=\c,dashed] (-0.9,-3.1) .. controls +(0.6,-0.3) and +(-0.6,-0.3) .. (0.9,-3.1);
 \draw[rotate=\t,color=\c] (-0.9,-3.1) .. controls +(0.6,0.3) and +(-0.6,0.3) .. (0.9,-3.1);
 \foreach \s in {-1,1} 
 \draw[rotate=\t,color=\c,xscale=\s] (0,-3.8) .. controls +(1.6,0) and +(0.1,-0.5) .. (2.6,-2.2) .. controls +(-0.05,0.25) and +(0.4,0.1) .. (2,-2) .. controls +(-0.6,-0.15) and +(0.8,0) .. (0,-2.5);
 }
 \foreach \t in {0,60,120,180} {
 \draw[rotate=\t,color=orange] (-7,0) .. controls +(0.4,-0.2) and +(-0.4,-0.2) .. (-5.6,0);
 \draw[rotate=\t,color=orange,dashed] (-7,0) .. controls +(0.4,0.2) and +(-0.4,0.2) .. (-5.6,0);}
 \foreach \t/\c in {60/purple,120/purple,240/blue,300/blue} {
 \draw[rotate=\t,color=\c] (-5.2,0) .. controls +(0.4,-0.2) and +(-0.4,-0.2) .. (-3.3,0);
 \draw[rotate=\t,color=\c,dashed] (-5.2,0) .. controls +(0.4,0.2) and +(-0.4,0.2) .. (-3.3,0);
 \foreach \s in {-1,1} 
 \draw[rotate=\t,color=\c,yscale=\s] (-6.5,0) .. controls +(0,-1) and +(-0.5,0) .. (-5.5,-1.5) .. controls +(1,0) and +(-1,0) .. (-3.05,-1.25) .. controls +(0.3,0) and +(0,-0.6) .. (-2.4,0);}
\end{scope}
\begin{scope} [scale=0.54,xshift=-6cm]
 \draw (0,0) circle (7);
 \draw (-1,0) ..controls +(0.5,-0.4) and +(-0.5,-0.4) .. (1,0) (-0.8,-0.1) ..controls +(0.6,0.3) and +(-0.6,0.3) .. (0.8,-0.1);
 \foreach \t in {60,150,210,300} {
 \draw[rotate=\t] (-1,-5.3) ..controls +(0.5,-0.4) and +(-0.5,-0.4) .. (1,-5.3) (-0.8,-5.4) ..controls +(0.6,0.3) and +(-0.6,0.3) .. (0.8,-5.4);}
 \foreach \t in {30,120,240,330} {
 \draw[rotate=\t] (-1,-3) ..controls +(0.5,-0.4) and +(-0.5,-0.4) .. (1,-3) (-0.8,-3.1) ..controls +(0.6,0.3) and +(-0.6,0.3) .. (0.8,-3.1);}
 \foreach \s in {-1,1}
 \draw[xscale=\s] (2.7,1.5) node {$\scriptstyle{*}$};
 \foreach \x/\c in {-0.1/orange,0.1/purple} {
 \draw[color=\c] (\x,-0.3) .. controls +(-0.3,-1) and +(-0.3,1) .. (\x,-7);
 \draw[color=\c,dashed] (\x,-0.3) .. controls +(0.3,-1) and +(0.3,1) .. (\x,-7);}
 \foreach \x/\c in {-0.1/blue,0.1/green} {
 \draw[color=\c] (\x,0.12) .. controls +(0.3,1) and +(0.3,-1) .. (\x,7);
 \draw[color=\c,dashed] (\x,0.12) .. controls +(-0.3,1) and +(-0.3,-1) .. (\x,7);} 
 \foreach \t/\c in {150/green,210/green,-60/blue,60/blue}
 \draw[rotate=\t,color=\c] (0,-5.4) ellipse (1.2 and 0.6);
 \foreach \t/\c in {30/orange,330/orange,120/purple,240/purple}
 \draw[rotate=\t,color=\c] (0,-3.1) ellipse (1.1 and 0.4);
 \foreach \s in {-1,1} {
 \draw[xscale=\s,color=green] (0.4,0.04) .. controls +(0.8,0.1) and +(-0.2,-0.6) .. (2.5,1.4);
 \draw[xscale=\s,color=green,dashed] (0.4,0.04) .. controls +(0.2,0.5) and +(-0.4,-0.1) .. (2.5,1.4);}
 \draw[color=orange,dashed] (-2,5.05) .. controls +(0.6,0.3) and +(-0.6,0.3) .. (2,5.05);
 \draw[color=orange] (-2,5.05) .. controls +(0.6,-0.3) and +(-0.6,-0.3) .. (2,5.05);
 \foreach \s in {-1,1} 
 \draw[color=orange,xscale=\s] (0,6.4) .. controls +(2,0) and +(0.2,1) .. (4,4.3) .. controls +(-0.1,-0.5) and +(0.8,-0.2) .. (2.7,3.8) .. controls +(-1.2,0.3) and +(1,0) .. (0,4.5);
 \draw[color=orange,dashed] (-2.3,2) .. controls +(0.6,0.3) and +(-0.6,0.3) .. (2.3,2);
 \draw[color=orange] (-2.3,2) .. controls +(0.6,-0.3) and +(-0.6,-0.3) .. (2.3,2);
 \foreach \s in {-1,1} 
 \draw[color=orange,xscale=\s] (0,3.2) .. controls +(2,0) and +(0,1) .. (3.5,1.1) .. controls +(0,-0.7) and +(0.4,-0.1) .. (2.7,0.5) .. controls +(-1.2,0.3) and +(1,0) .. (0,1.3);
 \draw[color=green,dashed] (-0.9,-3.1) .. controls +(0.6,-0.3) and +(-0.6,-0.3) .. (0.9,-3.1);
 \draw[color=green] (-0.9,-3.1) .. controls +(0.6,0.3) and +(-0.6,0.3) .. (0.9,-3.1);
 \foreach \s in {-1,1} {
 \draw[color=green,xscale=\s] (0,-3.8) .. controls +(1.6,0) and +(0.1,-0.5) .. (2.6,-2.2) .. controls +(-0.05,0.25) and +(0.4,0.1) .. (2,-2) .. controls +(-0.6,-0.15) and +(0.8,0) .. (0,-2.5) 
 (0,-5) .. controls +(2,0) and +(-1,-1) .. (4.3,-3.4)
 (0,-5.3) .. controls +(5,0) and +(1.8,-0.6) .. (5.2,-1.4) .. controls +(-1.2,0.4) and +(4,0) .. (0,-4.4);
 \draw[color=green,xscale=\s,dashed] 
 (0,-4.7) .. controls +(2,0) and +(-1,-0.5) .. (4.3,-3.4);
 }
 \foreach \t in {30,150} {
 \draw[rotate=\t,color=orange] (-7,0) .. controls +(0.4,-0.2) and +(-0.4,-0.2) .. (-5.6,0);
 \draw[rotate=\t,color=orange,dashed] (-7,0) .. controls +(0.4,0.2) and +(-0.4,0.2) .. (-5.6,0);}
 \foreach \s in {-1,1} {
 \draw[color=purple,xscale=\s] (1.7,-2.8) .. controls +(0.4,-0.3) and +(-0.4,-0.3) .. (4.4,-2.8)
 (4.8,-2.2) .. controls +(0.2,2) and +(0.7,-1) .. (3.2,4.3);
 \draw[color=purple,xscale=\s,dashed] (1.7,-2.8) .. controls +(0.4,0.3) and +(-0.4,0.3) .. (4.4,-2.8)
 (4.8,-2.2) .. controls +(-0.5,1.5) and +(0,-1) .. (3.2,4.3);
 }
 \foreach \s in {-1,1} {
 \draw[color=blue,xscale=\s] (2.8,1.7) .. controls +(0.4,0.5) and +(0.4,-0.3) .. (2.8,4.4)
 (2,-2.3) .. controls +(0.4,0.5) and +(0.2,-0.3) .. (2.9,1);
 \draw[color=blue,xscale=\s,dashed] (2.8,1.7) .. controls +(-0.3,0.5) and +(-0.3,-0.3) .. (2.8,4.4)
 (2,-2.3) .. controls +(-0.1,0.6) and +(-0.2,-0.3) .. (2.9,1);
 }
 \foreach \s in {-1,1} {
 \draw[color=blue,xscale=\s] (4.25,-5.06) arc (-50:70:6.6) .. controls +(-1,0.4) and +(0,1) .. (1,2.5) .. controls +(0,-1) and +(0,1) .. (1.5,0) .. controls +(0,-1) and +(0,1) .. (0.5,-3) .. controls +(0,-2) and +(-1,-0.7) .. (4.25,-5.06);
 \draw[color=purple,xscale=\s] (4.37,-5.21) arc (-50:70:6.8) .. controls +(-1.2,0.5) and +(0,1.5) .. (0.8,2.5) .. controls +(0,-1) and +(0,1) .. (1.3,0) .. controls +(0,-1) and +(0,1) .. (0.4,-3) .. controls +(0,-2.5) and +(-1,-0.7) .. (4.37,-5.21);}
\end{scope}
\begin{scope} [scale=0.54,xshift=8.3cm]
 \draw (0,0) circle (7);
 \draw (-1,0) ..controls +(0.5,-0.4) and +(-0.5,-0.4) .. (1,0) (-0.8,-0.1) ..controls +(0.6,0.3) and +(-0.6,0.3) .. (0.8,-0.1);
 \foreach \t in {60,150,210,300} {
 \draw[rotate=\t] (-1,-5.3) ..controls +(0.5,-0.4) and +(-0.5,-0.4) .. (1,-5.3) (-0.8,-5.4) ..controls +(0.6,0.3) and +(-0.6,0.3) .. (0.8,-5.4);}
 \foreach \t in {30,330} {
 \draw[rotate=\t] (-1,-3) ..controls +(0.5,-0.4) and +(-0.5,-0.4) .. (1,-3) (-0.8,-3.1) ..controls +(0.6,0.3) and +(-0.6,0.3) .. (0.8,-3.1);}
 \draw[rotate=180] (-1,-2) ..controls +(0.5,-0.4) and +(-0.5,-0.4) .. (1,-2) (-0.8,-2.1) ..controls +(0.6,0.3) and +(-0.6,0.3) .. (0.8,-2.1);
 \foreach \s in {-1,1}
 \draw (-2.7,4.65) node {$\scriptstyle{*}$} (-1.5,-2.7) node {$\scriptstyle{*}$};
 \foreach \x/\c in {-0.1/orange,0/blue,0.1/purple} {
 \draw[color=\c] (\x,-0.3) .. controls +(-0.3,-1) and +(-0.3,1) .. (\x,-7);
 \draw[color=\c,dashed] (\x,-0.3) .. controls +(0.3,-1) and +(0.3,1) .. (\x,-7);}
 \draw[color=green,rotate=60] (-0.1,0.33) .. controls +(0.3,1) and +(0.3,-1) .. (0,7);
 \draw[color=green,rotate=60,dashed] (-0.1,0.33) .. controls +(-0.3,1) and +(-0.3,-1) .. (0,7);
 \foreach \t/\c in {150/green,210/green,-60/blue,60/blue}
 \draw[rotate=\t,color=\c] (0,-5.4) ellipse (1.2 and 0.6);
 \foreach \t in {30,330}
 \draw[rotate=\t,color=orange] (0,-3.1) ellipse (1.1 and 0.4);
 \draw[color=orange] (0,2.1) ellipse (1.2 and 0.5);
 \draw[color=purple] (0,2.1) ellipse (1.1 and 0.4); 
 \draw[color=green] (0,0.1) .. controls +(0.2,0.6) and +(0.2,-0.6) .. (0,1.9);
 \draw[color=green,dashed] (0,0.1) .. controls +(-0.2,0.6) and +(-0.2,-0.6) .. (0,1.9);
 \draw[color=orange,dashed] (-2,5.05) .. controls +(0.6,0.3) and +(-0.6,0.3) .. (2,5.05);
 \draw[color=orange] (-2,5.05) .. controls +(0.6,-0.3) and +(-0.6,-0.3) .. (2,5.05);
 \foreach \s in {-1,1} 
 \draw[color=orange,xscale=\s] (0,6.4) .. controls +(2,0) and +(0.2,1) .. (4,4.3) .. controls +(-0.1,-0.5) and +(0.8,-0.2) .. (2.7,3.8) .. controls +(-1.2,0.3) and +(1,0) .. (0,4.5);
 \draw[color=green,dashed] (-0.9,-3.1) .. controls +(0.6,-0.3) and +(-0.6,-0.3) .. (0.9,-3.1);
 \draw[color=green] (-0.9,-3.1) .. controls +(0.6,0.3) and +(-0.6,0.3) .. (0.9,-3.1);
 \foreach \s in {-1,1} {
 \draw[color=green,xscale=\s] (0,-3.8) .. controls +(1.6,0) and +(0.1,-0.5) .. (2.6,-2.2) .. controls +(-0.05,0.25) and +(0.4,0.1) .. (2,-2) .. controls +(-0.6,-0.15) and +(0.8,0) .. (0,-2.5) 
 (0,-5) .. controls +(2,0) and +(-1,-1) .. (4.3,-3.4)
 (0,-5.3) .. controls +(5,0) and +(1.8,-0.6) .. (5.2,-1.4) .. controls +(-1.2,0.4) and +(4,0) .. (0,-4.4);
 \draw[color=green,xscale=\s,dashed] 
 (0,-4.7) .. controls +(2,0) and +(-1,-0.5) .. (4.3,-3.4);
 }
 \foreach \t in {30,150} {
 \draw[rotate=\t,color=orange] (-7,0) .. controls +(0.4,-0.2) and +(-0.4,-0.2) .. (-5.6,0);
 \draw[rotate=\t,color=orange,dashed] (-7,0) .. controls +(0.4,0.2) and +(-0.4,0.2) .. (-5.6,0);}
 \foreach \s in {-1,1} {
 \draw[color=purple,xscale=\s] (1.7,-2.8) .. controls +(0.4,-0.3) and +(-0.4,-0.3) .. (4.4,-2.8)
 (4.8,-2.2) .. controls +(0.2,2) and +(0.7,-1) .. (3.2,4.3);
 \draw[color=purple,xscale=\s,dashed] (1.7,-2.8) .. controls +(0.4,0.3) and +(-0.4,0.3) .. (4.4,-2.8)
 (4.8,-2.2) .. controls +(-0.5,1.5) and +(0,-1) .. (3.2,4.3);
 }
 \foreach \s in {-1,1} {
 \draw[color=blue,xscale=\s] (0.6,2.2) .. controls +(0.6,0.2) and +(0.1,-0.4) .. (2.8,4.4)
 (2.15,-2.3) .. controls +(0.1,0.6) and +(0.4,-0.4) .. (0.6,2);
 \draw[color=blue,xscale=\s,dashed] (0.6,2.2) .. controls +(0,0.5) and +(-0.4,-0.1) .. (2.8,4.4)
 (2.15,-2.3) .. controls +(-0.2,0.3) and +(-0,-0.3) .. (0.6,2);
 }
 \foreach \s in {-1,1} {
 \draw[color=purple,xscale=\s] (4.25,-5.06) arc (-50:70:6.6) .. controls +(-1.2,0.4) and +(0,1) .. (1.5,2.5) .. controls +(0,-1) and +(0,1) .. (1.5,0) .. controls +(0,-1) and +(0,1) .. (0.5,-3) .. controls +(0,-2) and +(-1,-0.7) .. (4.25,-5.06);
 \draw[color=blue,xscale=\s] (4.37,-5.21) arc (-50:90:6.8) (0,1) arc (90:0:1.2) .. controls +(0,-1) and +(0,1) .. (0.4,-3) .. controls +(0,-2.5) and +(-1,-0.7) .. (4.37,-5.21);}
\end{scope}
\begin{scope} [scale=0.54,xshift=10cm,yshift=13.5cm]
 \draw (0,0) circle (6);
 \draw (-1,0) ..controls +(0.5,-0.4) and +(-0.5,-0.4) .. (1,0) (-0.8,-0.1) ..controls +(0.6,0.3) and +(-0.6,0.3) .. (0.8,-0.1);
 \foreach \t in {30,90,270,330} 
 \draw[rotate=\t] (-1,-4) ..controls +(0.5,-0.4) and +(-0.5,-0.4) .. (1,-4) (-0.8,-4.1) ..controls +(0.6,0.3) and +(-0.6,0.3) .. (0.8,-4.1);
 \draw[rotate=180] (-1,-2) ..controls +(0.5,-0.4) and +(-0.5,-0.4) .. (1,-2) (-0.8,-2.1) ..controls +(0.6,0.3) and +(-0.6,0.3) .. (0.8,-2.1) (-1,-4.3) ..controls +(0.5,-0.4) and +(-0.5,-0.4) .. (1,-4.3) (-0.8,-4.4) ..controls +(0.6,0.3) and +(-0.6,0.3) .. (0.8,-4.4);
 \foreach \x/\c in {-0.1/blue,0.1/purple} {
 \draw[color=\c] (\x,-0.3) .. controls +(-0.3,-1) and +(-0.3,1) .. (\x,-6);
 \draw[color=\c,dashed] (\x,-0.3) .. controls +(0.3,-1) and +(0.3,1) .. (\x,-6);}
 \foreach \c/\t/\s in {orange/-60/1,green/60/-1} {
 \draw[color=\c,rotate=\t] (0.1*\s,0.33) .. controls +(0.3,1) and +(0.3,-1) .. (0,6);
 \draw[color=\c,rotate=\t,dashed] (0.1*\s,0.33) .. controls +(-0.3,1) and +(-0.3,-1) .. (0,6); } 
 \foreach \t/\c in {90/orange,30/blue,330/blue}
 \draw[rotate=\t,color=\c] (0,-4.1) ellipse (1.2 and 0.6);
 \foreach \y/\c in {2.1/orange,4.4/green} 
 \draw[color=\c] (0,\y) ellipse(1.2 and 0.5);
 \draw[color=purple] (0,2.1) ellipse (1.1 and 0.4); 
 \draw[color=green] (0,0.1) .. controls +(0.2,0.6) and +(0.2,-0.6) .. (0,1.9);
 \draw[color=green,dashed] (0,0.1) .. controls +(-0.2,0.6) and +(-0.2,-0.6) .. (0,1.9);
 \draw[color=blue] (0,2.3) .. controls +(0.2,0.6) and +(0.2,-0.6) .. (0,4.18);
 \draw[color=blue,dashed] (0,2.3) .. controls +(-0.2,0.6) and +(-0.2,-0.6) .. (0,4.18);
 \foreach \t/\c in {0/green,60/purple,300/purple} {
 \draw[rotate=\t,color=\c,dashed] (-1.3,-3.95) .. controls +(0.6,-0.3) and +(-0.6,-0.3) .. (1.3,-3.95);
 \draw[rotate=\t,color=\c] (-1.3,-3.95) .. controls +(0.6,0.3) and +(-0.6,0.3) .. (1.3,-3.95);}
 \foreach \t/\c in {0/green,300/purple} {
 \foreach \s in {-1,1} 
 \draw[rotate=\t,color=\c,xscale=\s] (0,-5.2) .. controls +(2,0) and +(0.2,-1) .. (3.5,-3) .. controls +(-0.1,0.5) and +(0.4,0.1) .. (2.5,-2.5) .. controls +(-1.2,-0.3) and +(1,0) .. (0,-3.3);}
 \draw[color=green] (-3.9,-0.3) .. controls +(2,-2) and +(-2,-2) .. (3.9,-0.3);
 \draw[color=green,dashed] (-3.9,-0.3) .. controls +(2,-1.5) and +(-2,-1.5) .. (3.9,-0.3);
 \foreach \s in {-1,1} 
 \draw[color=green,xscale=\s] (0,-2.6) .. controls +(2,0) and +(0,-2) .. (4.9,0) .. controls +(0,0.7) and +(0.5,0) .. (4.1,1.4) .. controls +(-1,0) and +(2,0) .. (0,-1);
 \foreach \s/\c in {1/purple,-1/orange} {
 \draw[xscale=\s,color=\c] (0.8,4.4) .. controls +(0.2,-1) and +(-1,0.2) .. (4.1,0.8);
 \draw[xscale=\s,color=\c,dashed] (0.8,4.4) .. controls +(1,-0.5) and +(-0.5,1) .. (4.1,0.8);}
 \draw[color=orange] (-4.6,0) .. controls +(0,-0.5) and +(-0.4,0) .. (-4,-1.2) .. controls +(0.3,0) and +(0,-0.3) .. (-3.5,-0.5) .. controls +(0,1) and +(-1,-0.2) .. (0,3.5) .. controls +(0.5,0.1) and +(0,-0.4) .. (1.4,4.4) .. controls +(0,0.4) and +(0.7,-0.1) .. (0,5.2) .. controls +(-2.1,0.3) and +(0,2) .. (-4.6,0);
 \draw[color=purple] (3.4,-0.5) .. controls +(0,1) and +(1,-0.2) .. (0,3.5) .. controls +(-0.5,0.1) and +(0,-0.4) .. (-1.4,4.4) .. controls +(0,0.4) and +(-0.7,-0.1) .. (0,5.2) .. controls +(2.1,0.3) and +(0,2) .. (5.1,0) .. controls +(0,-2) and +(1.2,0.6) .. (2.4,-4.8) .. controls +(-0.8,-0.4) and +(0,-0.7) .. (0.8,-4) .. controls +(0,1) and +(0,-1) .. (3.4,-0.5);
 \foreach \s in {1,-1} {
 \draw[xscale=\s,color=blue] (0.8,2.1) .. controls +(0.3,-0.7) and +(-1,0) .. (4,0.6);
 \draw[xscale=\s,color=blue,dashed] (0.8,2.1) .. controls +(0.8,0) and +(-0.5,0.5) .. (4,0.6);}
 \foreach \s in {1,-1} 
 \draw[xscale=\s,color=blue] (0,5.5) .. controls +(3,0) and +(0,3) .. (5.5,0) .. controls +(0,-0.8) and +(0.8,0) .. (4,-1.5) .. controls +(-1.2,0) and +(2,0) .. (0,1);
 \foreach \t in {60,120} {
 \draw[rotate=\t,color=orange] (-6,0) .. controls +(0.4,-0.2) and +(-0.4,-0.2) .. (-4.3,0);
 \draw[rotate=\t,color=orange,dashed] (-6,0) .. controls +(0.4,0.2) and +(-0.4,0.2) .. (-4.3,0);}
\end{scope}
\end{tikzpicture}
\end{center}
\caption{Trisection diagram of $S^2\times S^2$ and quadrisection diagrams of $S^2\times S^2\times S^1$}
\label{figS2xS2xS1}
\end{figure}

We shall explain on the example of $\CP^2\times S^1$ how to draw a diagram of the multisection we have constructed. We start with the diagram of $\CP^2$ given in the left hand side of Figure~\ref{figCP2S1scheme}, whose surface is denoted by $S$, where the red curve represents $X_{12}$, the blue curve represents $X_{23}$ and the green curve represents $X_{13}$. We use the alternative scheme given in the right hand side of Figure~\ref{figCP2S1scheme} (this is just for convenience, because a nice picture came out with this scheme). We draw the surface $\Sigma$ as $N=6$ copies $S_i$ of $S$ set along a cycle and joined by tubes creating the hole in the middle, see Figure~\ref{figCP2xS1}. The $3$--dimensional handlebodies of the quadrisection are made of copies of the $3$--dimensional handlebodies of the trisection of $\CP^2$ and copies of the punctured surface $S$ product an interval, joined by tubes. For instance, in our example, the handlebody $W_{124}$ is made of:
\begin{itemize}\itemsep=0cm
 \item copies of $X_{13}$ represented by the purple curves on $S_1$ and $S_6$, 
 \item the product of a punctured $S$ with an interval running from $S_2$ to $S_3$, where the purple curves represent arcs on the punctured $S$ that define disks in the product: here we take any pair of disjoint properly embedded arcs in the punctured $S$ that cut it into a disk, their product with the interval are disks whose boundaries form a cut system of the associated $2$--handlebody,
 \item the same thing between $S_4$ and $S_5$,
 \item the tubes joining the above pieces, to which corresponds a meridian purple curve which could be drawn between $S_i$ and $S_{i+1}$ for any $i\neq2,4$.
\end{itemize}
The other cut-systems are obtained accordingly; the orange one represents $W_{123}$, the blue one $W_{134}$ and the green one $W_{234}$. If the monodromy $\varphi$ were non-trivial, then the green curves on $S_1/S_6$ would have to join arcs on $S_6$ to their images by $\varphi$ on $S_1$. Note that we have to represent the successive copies of $S$ with alternating orientation. The monodromy $\varphi$ reverses the orientation of $S$ precisely when $n$ is odd.

On this diagram, two destabilizations can be performed. Sliding the orange curve on $S_2$ along the orange curve on $S_1$ and the green curve on $S_2$ along the green curve on $S_3$, we get parallel green and purple curves dual to parallel blue and orange curves. Thanks to Proposition~\ref{prop:destabilization}, this allows to destabilize once. The second destabilization is symmetric with respect to a vertical axis.

Starting with the genus--$2$ trisection diagram of $S^2\times S^2$ given in Figure~\ref{figS2xS2xS1}, we get a quadrisection diagram of $S^2\times S^2\times S^1$ in a completely analogous manner. In particular, we use the same scheme (see Figure~\ref{figCP2S1scheme}) and the same colors. This first gives the genus--$13$ quadrisection diagram on top-left of Figure~\ref{figS2xS2xS1}. Performing four destabilizations similar to that of Figure~\ref{figCP2xS1}, we get the genus--$9$ quadrisection diagram on Figure~\ref{figS2xS2xS1}. We can then perform two more destabilizations, which successively merge the pairs of ``holes'' marked with $*$: at the first step we readily have a pair of parallel blue and purple curves, and we slide some green curves to get the destabilization conditions; at the second step we slide the green curve around the upper $*$--marked hole along the similar green curve on the right, and we slide blue and purple curves to get parallel curves joining the $*$--marked holes. Finally, we get the genus--$7$ diagram on top-right of Figure~\ref{figS2xS2xS1}.

Note that we obtain a genus--$7$ quadrisection, which is stricly less than $9=ng+1$. This might not be surprising since the last two destabilizations seems to occur by chance more than following a general pattern.

\def\cprime{$'$}
\providecommand{\bysame}{\leavevmode ---\ }
\providecommand{\og}{``}
\providecommand{\fg}{''}
\providecommand{\smfandname}{\&}
\providecommand{\smfedsname}{\'eds.}
\providecommand{\smfedname}{\'ed.}
\providecommand{\smfmastersthesisname}{M\'emoire}
\providecommand{\smfphdthesisname}{PhD thesis}


\begin{thebibliography}{BCGM23}

\bibitem[BCGM23]{BCGM}
{\scshape F.~{Ben Aribi}, S.~Courte, M.~Golla {\normalfont \smfandname}
  D.~Moussard} -- {\og Multisections of higher-dimensional manifolds\fg},
  arXiv:2303.08779, 2023.

\bibitem[GK16]{GayKirby}
{\scshape D.~T. Gay {\normalfont \smfandname} R.~Kirby} -- {\og Trisecting
  4--manifolds\fg}, \emph{Geometry \& Topology} \textbf{20} (2016), no.~6,
  p.~3097--3132.

\bibitem[IN24]{IN}
{\scshape G.~Islambouli {\normalfont \smfandname} P.~Naylor} -- {\og
  Multisections of 4--manifolds\fg}, \emph{Transactions of the American
  Mathematical Society} \textbf{377} (2024), no.~2, p.~1033--1068.

\bibitem[Joh95]{Johannson}
{\scshape K.~Johannson} -- \emph{Topology and combinatorics of 3-manifolds},
  Lecture Notes in Mathematics, vol. 1599, Springer-Verlag, Berlin, 1995.

\bibitem[Koe21]{Koenig}
{\scshape D.~Koenig} -- {\og Trisections of 3--manifold bundles over
  {$S^1$}\fg}, \emph{Algebraic \& Geometric Topology} \textbf{21} (2021),
  no.~6, p.~2677--2702.

\bibitem[RS82]{RS}
{\scshape C.~P. Rourke {\normalfont \smfandname} B.~J. Sanderson} --
  \emph{Introduction to piecewise-linear topology}, Springer Study Edition,
  1982.

\bibitem[RT20]{RuTi}
{\scshape J.~H. Rubinstein {\normalfont \smfandname} S.~Tillmann} -- {\og
  Multisections of piecewise linear manifolds\fg}, \emph{Indiana University
  Mathematics Journal} \textbf{69} (2020), no.~6, p.~2208--2238.

\bibitem[Wil20]{Williams}
{\scshape M.~Williams} -- {\og Trisections of flat surface bundles over
  surfaces\fg}, \smfphdthesisname, The University of Nebraska--Lincoln, 2020.

\end{thebibliography}
\end{document}